\documentclass[12pt,reqno]{amsart}

\usepackage[margin=1in]{geometry}
\usepackage{graphicx}
\usepackage{color}
\usepackage{amsthm}
\usepackage{hyperref}
\usepackage{amssymb}
\usepackage[capitalise]{cleveref}
\usepackage{cite}
\usepackage{bm}

\usepackage{mathtools}
\DeclarePairedDelimiter{\ceil}{\lceil}{\rceil}

\newtheorem{theorem}{Theorem}[section]	
\newtheorem{corollary}[theorem]{Corollary}
\newtheorem{lemma}[theorem]{Lemma}
\newtheorem{proposition}[theorem]{Proposition}

\theoremstyle{definition}
\newtheorem{definition}[theorem]{Definition}

\theoremstyle{remark}
\newtheorem{remark}[theorem]{Remark}
\newtheorem{example}[theorem]{Example}

\numberwithin{equation}{section}

\newcommand{\form}{\mathfrak q}
\newcommand{\gr}{\Psi}
\newcommand{\smin}{\sigma_-}
\newcommand{\smax}{\sigma_+}

\newcommand{\cB}{{\mathcal{B}}}

\newcommand{\cK}{{\mathcal{K}}}
\newcommand{\cL}{{\mathcal{L}}}
\newcommand{\cN}{{\mathcal{N}}}
\newcommand{\cM}{{\mathcal{M}}}

\newcommand{\cF}{{\mathcal{F}}}
\newcommand{\cH}{{\mathcal{H}}}
\newcommand{\cV}{{\mathcal{V}}}

\newcommand{\C}{{\mathbb{C}}}
\newcommand{\R}{{\mathbb{R}}}
\newcommand{\Z}{{\mathbb{Z}}}
\newcommand{\N}{{\mathbb{N}}}
\newcommand{\bbR}{{\mathbb{R}}}
\newcommand{\bbC}{{\mathbb{C}}}

\newcommand\zbar{\overline{z}}
\newcommand{\ran}{\text{\rm{ran}}}

\newcommand{\Lframe}[2]{{\begin{pmatrix}  #1 \\ #2 \end{pmatrix}}}

\usepackage[dvipsnames]{xcolor}

\newcommand\term[1]{\emph{#1}}  
\newcommand{\spec}{\operatorname{Spec}}
\newcommand\spess{\spec_{\mathrm{ess}}}

\renewcommand{\det}{\operatorname{det}}

\newcommand{\dom}{\operatorname{dom}}
\newcommand{\codim}{\operatorname{codim}}

\newcommand{\tr}{\mathbf{\Gamma}}

\newcommand{\Ran}{\operatorname{ran}}

\DeclareMathOperator{\GL}{GL}
\newcommand{\minop}{S}

\newcommand{\lb}{\label}

\renewcommand{\ker}{\operatorname{ker}}

\newcommand{\Lagr}{\Lambda}  

\DeclareMathOperator\iD{\iota}
\DeclareMathOperator{\mul}{mul}
\DeclareMathOperator{\rank}{rank}
\DeclareMathOperator{\Mas}{Mas}

\DeclareMathOperator{\Span}{span}
\DeclareMathOperator{\deff}{def}

\date{\today}
\begin{document}

\title[The Duistermaat index and eigenvalue interlacing]
{The Duistermaat index and eigenvalue interlacing for self-adjoint
  extensions of a symmetric operator}

\author[G. Berkolaiko]{Gregory Berkolaiko}
\author[G. Cox]{Graham Cox}
\author[Y. Latushkin]{Yuri Latushkin}
\author[S. Sukhtaiev]{Selim Sukhtaiev}

\begin{abstract}
  Eigenvalue interlacing is a useful tool in linear algebra and
  spectral analysis.  In its simplest form, the interlacing inequality
  states that a rank-one positive perturbation shifts each eigenvalue
  up, but not further than the next unperturbed eigenvalue.  For
  different types of perturbations this idea is known as Weyl
  interlacing, Cauchy interlacing, Dirichlet--Neumann bracketing and
  so on.

  We prove a sharp version of the interlacing inequalities for
  ``finite-dimensional perturbations in boundary conditions,''
  expressed as bounds on the spectral shift between two self-adjoint
  extensions of a fixed semibounded symmetric operator with finite and
  equal defect numbers.  The bounds are given in terms of the
  Duistermaat index, a topological invariant describing the relative
  position of three Lagrangian planes in a symplectic space.  Two of
  the Lagrangian planes describe the self-adjoint extensions being
  compared, while the third corresponds to the Friedrichs extension,
  which acts as a reference point.

  Along the way several auxiliary results are established, including
  one-sided continuity properties of the Duistermaat index,
  smoothness of the Cauchy data space without unique
  continuation-type assumptions, and a formula for the Morse index 
  of an extension of a non-negative symmetric operator.
\end{abstract}


\maketitle

\section{Introduction}\lb{intro}

\subsection{Background and motivation}
Let $H_1$ and $H_2$ be two $N\times N$ Hermitian matrices, with 
eigenvalues $\lambda_1(H_j) \leq \lambda_2(H_j) \leq \cdots \leq \lambda_N(H_j)$.
The following inequalities \cite[Cor.~4.3.3]{HornJohnson} are often
referred to as ``Weyl Interlacing'':
\begin{equation}
  \label{eq:WeylInterlacingMatrices}
  \lambda_{k-\smin}(H_1)
  \leq
  \lambda_k(H_2)
  \leq
  \lambda_{k+\smax}(H_1),
\end{equation}
where $\smax$ and $\smin$ are given by the number of positive and
negative eigenvalues of the perturbation $H_2-H_1$.  The fact that
$\smin+\smax$ gives the rank of the perturbation from $H_1$ to $H_2$
suggests that the bounds in (\ref{eq:WeylInterlacingMatrices}) are
optimal.  Cauchy Interlacing
\cite[Cor.~4.3.17]{HornJohnson}, where $H_2$ is
obtained from $H_1$ by removing rows and columns, can also be put in
the form \eqref{eq:WeylInterlacingMatrices} by adjusting $\smax$ to
include the number of removed rows and columns.

The principal aim of this work is to establish inequality
\eqref{eq:WeylInterlacingMatrices} for any two self-adjoint extensions
of a bounded from below symmetric operator $\minop$ with finite and equal
defect numbers.  The results are directly applicable to differential
operators with finite-dimensional changes in boundary conditions, in
settings such as linear Hamiltonian systems
\cite{HowJunKwo_jdde18,HowSuk_dcds20}, quantum graphs
\cite{BanBerRazSmi_cmp12,BerKuc_incol12,BerKenKurMug_tams19}, \v{S}eba
billiards
\cite{Seb_prl90,BogGirSch_pre02,RahFis_n02,KeaMarWin_jmp10,KurUeb_gafa14,KurUeb_imrn23},
and manifolds with conical singularities \cite{GilKraMen_cjm07}; some
concrete examples are discussed in \Cref{sec:examples}.  Unlike the
matrix case discussed above, the perturbations here are not additive
(since one cannot take the difference of two unbounded operators with
different domains), so it is not immediately clear how to define such
quantities as the signature $(\smin,\smax)$ of the perturbation.

Our results characterize the shifts
in the interlacing, $\smin$ and $\smax$, in terms of the relative
topological position of three pieces of data: the two Lagrangian
planes that describe the self-adjoint extensions of interest and a third Lagrangian
plane describing the Friedrichs extension.
This topological position is expressed via the
Duistermaat triple index \cite{BarOffPorWu_mz21, Dui_am76,How_jmaa21,
  ZhoWuZhu_fmc18}, an integer-valued symplectic invariant whose
definition is recalled and supplemented with several new computational
tools in \Cref{sec:index_formulas}; see also \Cref{sec:examples} for
examples of computation.

For illustrative purposes, we present two proofs of our main result:
via Maslov-type index theory
\cite{Arn_faa67,BooZhu_memAMS,CapLeeMil_cpam94,RobSal_t93} and via the
Krein resolvent formula \cite{BehHasDeS_boundarytriples, MR18341,
  MR482343, MR1628437, MR203474, Schmudgen_unboundedSAO}, in
\Cref{sec:proof1,sec:proof_krein_resolvent}, correspondingly.  Along the
way, we sharpen existing techniques, in particular to avoid relying on
unique continuation-type conditions; see \Cref{sec:CauchyDataSpace}.

\subsection{Main results}
Let $\cH$ be a separable Hilbert space and let $\minop$ be a closed,
bounded from below, densely defined symmetric operator with finite and
equal defect numbers $(n,n)$.  Under these assumptions, all
self-adjoint extensions $H$ of $\minop$ have the same essential
spectrum\footnote{This coincides with $\sigma_{e3}(S)$ in
  \cite[Chap. IX]{EdmundsEvans_spectral}, so the first equality in
  \eqref{ess:same} follows from
  \cite[Cor. IX.4.2.]{EdmundsEvans_spectral}},
\begin{equation}
  \label{ess:same}
  \spess(H) = \spess(\minop)
  := \{z \in \bbC : \minop - z \text{ is not Fredholm} \}.
\end{equation}
We will be describing the self-adjoint extensions of $\minop$ in terms of a
boundary triplet $(\cK, \Gamma_0, \Gamma_1)$, see
\cite[Sec.~14.2]{Schmudgen_unboundedSAO}.  Here $\cK$ is an
$n$-dimensional\footnote{A boundary triplet exists if and only if the defect numbers of $\minop$ are equal, in which case $\dim \cK$ equals this common value, see \cite[Prop.~14.5]{Schmudgen_unboundedSAO}.} complex Hilbert space and the linear mappings $\Gamma_0, \Gamma_1 \colon \dom(\minop^*) \to \cK$ are such
that the operator
\begin{equation}
  \label{eq:trace_op}
	\tr \colon \dom(\minop^*) \to \cK\oplus\cK, \qquad  \tr f := (\Gamma_0f, \Gamma_1f)
\end{equation}
is surjective and the \emph{abstract Green's identity}
\begin{equation}
  \label{eq:Greens_identity}
  \left<f, \minop^* g\right>_\cH - \left<\minop^*f, g\right>_\cH
  = \left<\Gamma_0 f, \Gamma_1 g\right>_\cK - \left<\Gamma_1 f, \Gamma_0 g\right>_\cK
\end{equation}
holds for all $f,g \in \dom(\minop^*)$.

We will view $\cK \oplus \cK$ as
a complex symplectic space with the symplectic form
\begin{equation}
  \label{eq:sympl_form}
  \omega(u, v) := \left<u_0, v_1\right>_\cK
  - \left<u_1, v_0\right>_\cK,
  \quad
  u = (u_0, u_1) \text{ and }
 v = (v_0, v_1) \in \cK \oplus \cK,
\end{equation}
in terms of which the right-hand side of Green's identity
\eqref{eq:Greens_identity} is $\omega(\tr f, \tr g)$. Self-adjoint
extensions $H$ of $\minop$ are in one-to-one correspondence with
Lagrangian planes\footnote{We refer to
  Lagrangian subspaces as ``planes" regardless of their dimension. For
  a review of symplectic linear algebra in 
  complex vector spaces, we refer the reader to
  \cite{EveMar_tams99,Har_jpa00} or \cite[App.~D]{BerKuc_graphs}.}
$\cL$ in $\cK\oplus\cK$ via
\begin{equation}
  \label{eq:domain_extension}
  \dom(H) := \big\{f \in \dom(\minop^*) : (\Gamma_0 f, \Gamma_1 f)
  \in \cL\big\}.
\end{equation}
Heuristically this says that one must impose $\dim \cL = n$ ``boundary conditions" on $\minop^*$ to obtain 
a self-adjoint operator.

To state an analogue of inequalities
\eqref{eq:WeylInterlacingMatrices} for two self-adjoint extensions
$H_1$ and $H_2$ of $\minop$, in terms of the corresponding Lagrangian planes
$\cL_1$ and $\cL_2$, we will need a third
Lagrangian plane $\cF$, which corresponds to the Friedrichs extension
$H_F$ of $\minop$, namely
\begin{equation}
  \label{eq:Friedrichs_plane}
  \cF := \big\{ (\Gamma_0 f, \Gamma_1 f) : f \in \dom(H_F) \big\}.
\end{equation}
It is common in applications to choose the triple
$(\cK, \Gamma_0, \Gamma_1)$ so that the domain of the
Friedrichs extension is $\ker \Gamma_0$ and thus
$\cF$ coincides with the \emph{vertical subspace} $\cV := 0 \oplus \cK$. 
Some of the results below take a
simplified form under the assumption that $\cF = \cV$. 

Since $H$ is a self-adjoint extension, $\spec(H) \backslash \spess(H)$
consists of isolated eigenvalues of finite multiplicity.  For any
interval $I$ whose closure is disjoint from $\spess(H)$ we
define\footnote{The requirement that $\overline I$ lies outside the
  essential spectrum guarantees $N\big(H; I \big)$ is finite.}  the
\term{counting function}
\begin{equation}
  \label{eq:counting_def}
  N\big(H; I \big) := \sum_{\lambda\in I} \dim \ker (H-\lambda),
\end{equation}
i.e., the number of eigenvalues of $H$ in $I$ counted with
multiplicity.  We will abbreviate
\begin{align}
  \label{eq:n_plus_minus_def}
  n_-(H) := N\big(H; (-\infty,0)\big), \qquad
  n_0(H) := \dim \ker H
\end{align}
when these are well defined. The number 
$n_-(H)$ is called the \term{Morse index} of $H$.  If
$\lambda \in \bbR$ is below the essential spectrum, we also define the
\term{spectral shift}
\begin{equation}
  \label{eq:spec_shift_def}
  \sigma(H_1, H_2; \lambda) :=
  N\big(H_1;(-\infty,\lambda]\big) - N\big(H_2; (-\infty,\lambda]\big)
\end{equation}
between self-adjoint extensions $H_1$ and $H_2$.  Finally, assuming
$S$ is semibounded from below, we will label the eigenvalues of $H$
below the essential spectrum by
$\lambda_1(H) \leq \lambda_2(H) \leq \cdots$, i.e., in increasing
order, repeated according to their multiplicity.

We are now ready to formulate our main result.

\begin{theorem}
  \label{thm:main}
  Suppose $\minop$ is a closed, bounded from below, densely defined
  symmetric operator with finite and equal defect numbers and a
  boundary triplet $(\cK, \Gamma_0, \Gamma_1)$.  Let $\cL_1$, $\cL_2$
  and $\cF$ be Lagrangian planes in $(\cK\oplus\cK, \omega)$
  corresponding to self-adjoint extensions $H_1$, $H_2$ and the
  Friedrichs extension $H_F$ of $\minop$.  Define
  \begin{equation}
    \label{eq:sminmax}
    \smin := \iD(\cL_1, \cL_2, \cF),
    \qquad
    \smax := \iD(\cL_2, \cL_1, \cF),
  \end{equation}
  where $\iD$ is the Duistermaat index (see
  Section~\ref{sec:index_formulas}). Then the bounds
  \begin{equation}
    \label{eq:spec_shift_bound}
    -\smin \leq \sigma(H_1, H_2; \lambda ) \leq \smax
  \end{equation}
  hold for all $\lambda \in \bbR$ below 
  $\spess(\minop)$.  Equivalently, each of the inequalities
  \begin{equation}
    \label{eq:mainWeylInterlacing}
    \lambda_{k-\smin}(H_1)
    \leq
    \lambda_k(H_2)
    \leq
    \lambda_{k+\smax}(H_1)
  \end{equation}
  holds for all $k$ such that the eigenvalues in question are below
  $\spess(\minop)$.
\end{theorem}

It follows from elementary properties of the Duistermaat index that
\begin{equation}
  \label{eq:rank_formula}
  \smin + \smax = n - \dim(\cL_1\cap\cL_2),
\end{equation}
where $n=\dim\cL_1 = \dim\cL_2 = \dim \cK$ is the defect number of the
operator $S$.  Relation~\eqref{eq:rank_formula} gives a shortcut for
computing one index in \eqref{eq:sminmax} from the other.  The fact
that (\ref{eq:rank_formula}) is the rank of the perturbation from
$H_1$ to $H_2$ (in the sense of (\ref{eq:rank_res_diff}) below)
suggests that the bounds in \eqref{eq:spec_shift_bound} are optimal.
\Cref{prop:sharp_bounds} puts this on a rigorous footing: for any
$\smin$ and $\smax$ there are extensions $H_1$ and $H_2$ for which
both estimates in \eqref{eq:spec_shift_bound} are sharp.

The intuition behind needing $\cF$ is as follows.  The (complex)
Lagrangian Grassmannian $\Lambda$ --- the set of all Lagrangian planes
in $\cK\oplus\cK$ --- is diffeomorphic to the unitary group $U(n)$, a
compact manifold without boundary.  The plane $\cF$ provides a point
of reference in $\Lagr$, which allows us to establish facts on the
ordering of the eigenvalues in $\R$, such as
\eqref{eq:mainWeylInterlacing}. Heuristically, $H_F$ is the extension
of $\minop$ with the largest number of ``eigenvalues at
$\pm\infty$''.\footnote{See, for instance, \cite[Thm.~1.4]{Men_ijm15},
  which says a Lagrangian plane $\cL$ intersects $\cF$ nontrivially if
  and only if every neighborhood of $\cL$ contains an $\cL'$ whose
  self-adjoint extension $H_{\cL'}$ has eigenvalues close to
  $-\infty$.}
    
One approach to Theorem~\ref{thm:main} proceeds via the Maslov index
of a special path of Lagrangian planes, the \emph{Cauchy data space},
defined here with the parameter $z\in\bbC$ by
\begin{equation}
  \label{eq:CDS_def}
  \cM(z) := \big\{(\Gamma_0 f, \Gamma_1f): f\in\ker(\minop^* - z) \big\}
  \ \subset\  \cK \oplus \cK.
\end{equation}
Relevant properties of $\cM(z)$ are given in 
\Cref{prop:cM_real_s}.
Note that $\cM(z)$ is
the graph of the
Dirichlet-to-Neumann map (Weyl--Titschmarsh function) $M(z)$ when the
latter is defined.

\begin{theorem}
  \label{thm:shift_Hormander}
  Under the assumptions of \Cref{thm:main}, one has
  \begin{equation}
    \label{eq:count_triple}
    N\big(H_1; (a,b] \big) 
    - N\big(H_2; (a,b] \big) 
    \ =\
    \iD\!\big(\cL_1, \cL_2, \cM(b)\big) -   \iD\!\big(\cL_1, \cL_2, \cM(a)\big)
  \end{equation}
  for any $[a,b] \subset \bbR \backslash \spess(\minop)$. In particular, 
  the spectral shift is given by
  \begin{equation}
    \label{eq:shift_triple}
    \sigma(H_1, H_2; \lambda) 
    \ =\
    \iD\!\big(\cL_1, \cL_2, \cM(\lambda)\big) - \iD(\cL_1, \cL_2, \cF)
  \end{equation}
  for any $\lambda \in \bbR$ below the essential
  spectrum.
\end{theorem}

We remark that it is common to express eigenvalue counting functions
as Maslov indices (see, for instance, the survey \cite{Beck} and
references therein) but only under the assumption of the unique
continuation property (UCP) for the symmetric operator $\minop$: the
mapping $f \mapsto (\Gamma_0 f, \Gamma_1 f)$ is injective on
$\ker(\minop^*-z)$ for all $z \in \C$.  Without the UCP, the Maslov index may
miss some eigenvalues, as demonstrated in \Cref{sec:noUCP}.  A novel
feature of \Cref{thm:shift_Hormander} is that the \emph{UCP is
  not necessary when evaluating the spectral shift}.

Another remarkable feature of \Cref{thm:shift_Hormander} is that using
a Maslov-type index (or a spectral flow, or the Krein shift function)
invariably involves integration or evaluation over a path (for
example, continuously tracking a branch of the logarithm).  In
contrast, equation~\eqref{eq:count_triple} involves only the data
collected at the endpoints of the interval!

An easy corollary of \Cref{thm:shift_Hormander} 
is the following elegant formula for the
Morse index of an extension $H_\cL$ of a non-negative symmetric
operator $S$:
\begin{equation}
  \label{eq:Morse_Duistermaat_preview}
  n_-(H_\cL) = \iD\!\big( \cM(0), \cL, \cF \big).  
\end{equation}
We refer the reader to \Cref{thm:BL_formula} for a precise
formulation and references to related results.

Another approach to \Cref{thm:main} is via the resolvents of $H_1$ and
$H_2$, which can be compared using the Krein resolvent formula.
Throughout, we use the notation $\cN_z:=\ker(\minop^*-z)$, 
and denote the number of zero, positive and negative 
eigenvalues of a self-adjoint operator (whenever these quantities are finite) by 
$n_\bullet(\cdot)$ with $\bullet \in \{0, +, -\}$; cf. 
\eqref{eq:n_plus_minus_def}.

\begin{theorem}
  \label{thm:resolvent_diff}
  Assume the setting of Theorem \ref{thm:main}. For
  $\lambda\in \rho(H_1)\cap\rho(H_2) \cap \bbR$ we define the operator
  \begin{equation}
    \label{eq:resolvent_diff}
    D(\lambda) := (H_1-\lambda)^{-1} - (H_2-\lambda)^{-1}.
  \end{equation}
  Then one has
  \begin{align} 
    \begin{split}
    \label{eq:index_K-int}
    &n_-\big(D(\lambda)\big) = \iD\!\big(\cL_1, \cL_2, \cM(\lambda)\big),\\
    &n_+\big(D(\lambda)\big) = \iD\!\big(\cL_2, \cL_1, \cM(\lambda)\big),
    \end{split}\\
	&n_0\big(D(\lambda)|_{\cN_{\lambda}} \big)=\dim \cL_1\cap \cL_2.  \label{eq:0index_K-int}
\end{align}
In addition,  $D(\lambda)$ has constant rank given by
\begin{equation}
  \label{eq:rank_res_diff}
  \rank\big(D(\lambda)\big) = n-\dim \cL_1\cap \cL_2,
\end{equation} 
while the functions $\lambda\mapsto n_{\pm}(D(\lambda))$ are locally constant
on $\rho(H_1)\cap\rho(H_2) \cap \R$.
\end{theorem}

In \Cref{sec:proof_krein_resolvent} we will explain how \Cref{thm:main} can be obtained 
from \Cref{thm:resolvent_diff}. In particular, we will use the index formulas in \eqref{eq:index_K-int} 
to derive the interlacing inequality \eqref{eq:mainWeylInterlacing}.

\begin{remark}
  Assume the setting of Theorem \ref{thm:resolvent_diff}.
  \begin{enumerate}
  \item Since $H_1$ and $H_2$ are both extensions of $\minop$, we have $D(\lambda)f=0$ provided 
$f\in\cN_{\lambda}^{\perp}=\overline{\ran(\minop-\lambda)}$, thus $\cN_{\lambda}^{\perp} \subset \ker D(\lambda)$. In this context, the identity in \eqref{eq:0index_K-int} provides new information about the part of $\ker D(\lambda)$ contained in $\cN_\lambda$. 

\item The fact that the rank of $D(\lambda)$ is constant on $\rho(H_1)\cap\rho(H_2) \cap \bbR$ was shown in \cite[Thm.~2.8.1]{BehHasDeS_boundarytriples}; however, the explicit value $n-\dim \cL_1\cap\cL_2$ appears to be new.

  \item The indices $n_-\big(D(\lambda)\big)$ and
    $n_+\big(D(\lambda)\big)$ may change 
    when $\lambda$ passes through $\spec(H_1)\cup \spec(H_2)$,
    but their sum remains constant, as seen from
    \eqref{eq:rank_res_diff}, cf.\ \eqref{eq:rank_formula}.
    \hfill$\Diamond$
  

  \end{enumerate}
\end{remark}

\subsection{Related results and possible extensions}
\label{sec:literature_and_conjectures}

The most general prior result on eigenvalue interlacing known to us in
the context of self-adjoint extension is Theorem 10.2.5 of
\cite{BirmanSolomjak_book}.  This theorem only applies when the
operator $H_2$ is obtained from $H_1$ by the restriction of its form
domain.  For instance, no pair of extensions from
Example~\ref{sec:examples_n2} satisfies this condition.  In contrast,
our \Cref{thm:main} allows one to compare \emph{any} pair of
self-adjoint extensions of $S$.

Relations between the Morse index of self-adjoint extensions of $S$
and the boundary operators in $\cK\oplus\cK$ constructed by means of
the Weyl function (as in \Cref{thm:BL_formula}, for example) are
classical, and known at least since M.  Krein \cite{MR0024574}, Birman
\cite{MR0080271}, and Derkach--Malamud \cite{DerMal_jfa91}; see also
the recent treatise \cite{BehHasDeS_boundarytriples} and the
literature cited therein. However, the geometric approach via the
Duistermaat index offered in the current paper allow us, on the one
hand, to drop inessential restrictions (such as the
form domain inclusion condition in \cite[Thms.~5 and 6]{DerMal_jfa91})
and, on the other hand, to compute the spectral shift using the
Duistermaat index calculus that we review and extend in
Section~\ref{sec:index_formulas}.

The inequalities~\eqref{eq:spec_shift_bound}
and~\eqref{eq:mainWeylInterlacing} give different but
equivalent points of view on the same result; we include both of them
for completeness and ease of use in applications.  However, the
discrete spectral shift we consider is extended by a more flexible
concept of \emph{Krein spectral shift} \cite{MR1607900, MR0049490,
  MR1202723, MR0060742, MR0139006, Schmudgen_unboundedSAO}, valid also
on the continuous spectrum and in the gaps.  With this extension, we
conjecture that \eqref{eq:spec_shift_bound} holds for all real
$\lambda$. The case of a rank-one perturbation has been
thoroughly investigated in
\cite{BehLebMarMowTru_jmaa16,BehMowTru_om16}, in the more general
setting of self-adjoint operators on a Krein space.  The latter
results are formulated in the gaps of the essential spectrum (see also
\cite{BehMowTru_caot14}), which gives further support to our
conjecture about the universal validity of
\eqref{eq:spec_shift_bound}.  Note, additionally, that (most of) the conclusions
in \Cref{thm:shift_Hormander,thm:resolvent_diff} are already valid in the gaps.

\subsection*{Outline of paper}
In \Cref{sec:monotonicity} we review the crossing form and the computation of the Maslov index 
for monotone paths. In \Cref{sec:index_formulas} we recall the definition 
of the Duistermaat index, in addition to obtaining new results on its one-sided limits and explicit formulas 
for its calculation via Lagrangian frames. \Cref{sec:CauchyDataSpace} derives fundamental 
properties of the Cauchy data space, which we use in \Cref{sec:proof1}  to calculate its 
Maslov index and hence prove our main theorems. In \Cref{sec:proof_krein_resolvent} we 
give a second proof using the Krein resolvent formula, and in \Cref{sec:examples} 
we present some applications of our results. \Cref{app:Lagrangian} summarizes basic 
properties of the Lagrangian Grassmannian that are used throughout.

\subsection*{Notation and conventions}
We use $\langle \cdot,\,\cdot\rangle_{\cH}$ and
$\langle \cdot,\,\cdot\rangle_{\cK}$ to denote the scalar products on
$\cH$ on $\cK$.  
These are taken to be linear in the \emph{second}
argument, as is the symplectic form $\omega$ in
\eqref{eq:sympl_form}. We use $\oplus$ to denote the direct (not necessarily
orthogonal) sum, and the set of all Lagrangian planes in $(\cK \oplus \cK, \omega)$ is
denoted $\Lagr$. The zero subspace will be denoted by $0$ 
when the ambient space is clear from the context.
 The space of bounded linear operators between Hilbert spaces $\cH_1, \cH_2$ is denoted by $\cB(\cH_1, \cH_2)$. We denote by
$\spec(\cdot)$ and $\rho(\cdot)$ the spectrum and the resolvent set.


\subsection*{Acknowledgements}
The authors are grateful to Ram Band, Dean Baskin, Maurice de Gosson, Peter
Howard, Alexander V. Kiselev, Peter Kuchment, Marina Prokhorova, Gilad
Sofer and Igor Zelenko for their help, encouragement, and suggestions
of references.

G.B. was partially supported by NSF grant DMS-2247473.
G.C. acknowledges the support of NSERC Discovery Grant 2017-04259.
Y.L. was supported by the NSF grant DMS-2106157, and would like to
thank the Courant Institute of Mathematical Sciences at NYU and
especially Prof.\ Lai-Sang Young for their hospitality.  S.S. was
supported in part by NSF grants DMS--2510344, DMS--2418900, Simons Foundation grant
MP--TSM--00002897,grant no. $2024154$ from the U.S.-Israel Binational Science Foundation Jerusalem, Israel, and by the Office of the Vice President for
Research \& Economic Development (OVPRED) at Auburn University through
the Research Support Program grant.

The authors thank the anonymous referees for suggesting several
important corrections to the manuscript.

\section{Crossing forms and monotonicity}
\label{sec:monotonicity}

We first review crossing forms and the computation of the Maslov index for monotone paths. This  suffices for the purposes of this paper; see \Cref{app:Lagrangian} for a general discussion of the Maslov index and different parameterizations of Lagrangian subspaces. We work in a finite-dimensional complex symplectic space, $\cK \oplus \cK$, with Lagrangian Grassmannian $\Lambda$.

Let $\cL(\cdot)\colon [0,1]\to\Lagr$ be a differentiable path of Lagrangian planes. 
For $t_0 \in [0,1]$ and $u, v \in \cL(t_0)$  we define the crossing form $\form$ by
\begin{equation}
\label{Q:def}
	\form(u, v) := \frac{d}{dt} \omega\big(u, \tilde v(t) \big) \bigg|_{t=t_0},
\end{equation}
where $\tilde v(t)$ is any differentiable path in $\cK \oplus
\cK$ such that $\tilde v(t_0) = v$ and $\tilde v(t) \in \cL(t)$ for
all $t$ near $t_0$. For this to be a valid definition, we must show
that it does not depend on the choice of the path $\tilde v$.

Suppose $\tilde u(t)$ is a differentiable path in $\cK \oplus \cK$ with $\tilde
u(t_0) = u$ and $\tilde u(t) \in \cL(t)$. The fact that $\tilde u(t)$
and $\tilde v(t)$ are both in $\cL(t)$ implies $\omega\big(\tilde
u(t), \tilde v(t) \big) = 0$. Differentiating at $t_0$, we get
\begin{equation}
  \label{eq:sesqui_well_defined}
  \omega\big(u, \tilde v'(t_0)\big)
  = -\omega\big(\tilde u'(t_0), v\big)
  = \overline{ \omega\big(v, \tilde u'(t_0)\big) }.
\end{equation}
Since the left-hand side does not depend on $\tilde u(t)$ and
the right-hand side does not depend on $\tilde v(t)$, both sides are path independent. 
This proves that $\form(u,v)$ is well defined, and is equal to $\overline{\form(v,u)}$.

We now give some equivalent expressions for the quadratic form $\form[v] := \form(v,v)$. The analogous expressions for the sesquilinear form $\form(u,v)$ can be obtained by polarization.

\begin{theorem}\label{THMform}
  Let $\cL(\cdot)\colon [0,1]\to\Lagr$ be a differentiable path of Lagrangian planes.
  Fix $v \in \cL(t_0)$ and let $\hat\cL$ be a Lagrangian subspace transversal to $\cL(t_0)$.
  
  \begin{enumerate}
  \item If $w(t)$ is the unique path in $\hat\cL$ for which $v + w(t) \in \cL(t)$, then
    \begin{equation}
      \label{Q-RS}
      \form[v] = \frac{d}{dt} \omega\big(v, w(t) \big) \bigg|_{t=t_0}.
    \end{equation}
    
  \item If $\cL^\sharp$ is transversal to $\hat\cL$, $L_t \colon \cL^\sharp \to \hat\cL$ is a
    differentiable family of operators such that
    $\cL(t) = \{u+L_tu \colon u \in \cL^\sharp \}$ (i.e.,
    $\cL(t)$ is the graph of $L_t$) and $u \in \cL^\sharp$ is such that $v = u + L_{t_0} u$, then
    \begin{equation}
      \label{Q-St}
      \form[v] = \frac{d}{dt} \omega\big(u, L_t u \big) \bigg|_{t=t_0}.
    \end{equation}
    
  \item If
    $Z(t) = \Big(\! \begin{smallmatrix} X(t) \\ Y(t) \end{smallmatrix}
    \!\Big)$ is a differentiable frame for $\cL(t)$ and $\kappa \in \cK$
    is such that $v = Z(t_0) \kappa$, then
    \begin{equation}
      \label{Q;frame}
      \form[v] = \left<\kappa, \big(X^*(t_0) Y'(t_0) - Y^*(t_0) X'(t_0)\big) \kappa \right>_{\bbC^n}.
    \end{equation}
    
  \item If $G_t$ is a differentiable family of operators on
    $\cK \oplus \cK$ such that $G_t \cL(t_0) = \cL(t)$ and
    $G_{t_0} = I_{\cK \oplus \cK}$, then
    \begin{equation}
      \label{Q;Gt}
      \form[v] = \frac{d}{dt} \omega\big(v, G_t v \big) \bigg|_{t=t_0}.
    \end{equation}
  \end{enumerate}
\end{theorem}

Note that \eqref{Q-RS} is the definition of $\form$ given by Robbin and Salamon in \cite[Thm.~1.1]{RobSal_t93}, so this theorem shows that our definition in \eqref{Q:def} is equivalent to theirs. Similarly, choosing $\cL^\sharp = \cL(t_0)$ in \eqref{Q-St} recovers the definition in \cite[Eq.~(2.1)]{BooFur_tjm98}.

\begin{proof}
The given expressions for $\form[v]$ follow from \eqref{Q:def} with appropriate choices of the path $\tilde v(t)$, namely $v + w(t)$, $u + L_t u$, $Z(t)\kappa$ and $G_t v$. 
\end{proof}

The crossing form allows us to define a notion of monotonicity for differentiable paths.

\begin{definition}
  \label{def:increasing_path}
  A differentiable path $\cL(\cdot) \colon (0,1)\to\Lagr$ is
  \term{non-decreasing} (corresp.\ \term{increasing}) if the crossing form $\form = \form_{t_0}$ on $\cL(t_0)$ is
  non-negative (corresp.\ positive) at every $t_0$.
\end{definition}

The Maslov index of an increasing $C^1$ path 
$\cM(\cdot)$, with reference plane $\cL$, is given by
  \begin{equation}
    \label{eq:Maslov_intersections}
    \Mas_{[a,b]}\!\big(\cM(\cdot), \cL\big)
    = \sum_{t \in [a,b)} \dim \!\big(\cM(t)\cap\cL \big).
  \end{equation}
	 The fact that $\form$ is positive (and hence non-degenerate) guarantees that the crossings are isolated, 
  therefore the above sum is finite. This is a special case of  \cite[Prop.~3.27]{BooZhu_memAMS}; see 
  \eqref{MIforms} for the general formula. In practice we will use the equivalent formula
  \begin{equation}
    \label{eq:Maslov_intersections2}
    \Mas_{[a,b]}\!\big(\cL, \cM(\cdot)\big)
    = - \sum_{t \in (a,b]} \dim \!\big(\cM(t)\cap\cL \big).
  \end{equation}
  which is obtained from \eqref{eq:Maslov_intersections} using the identity \eqref{eq:CLM_antisymmetric}. Note that 
  the sum in \eqref{eq:Maslov_intersections} is over $t \in [a,b)$, whereas the sum in \eqref{eq:Maslov_intersections2} is over $t \in (a,b]$.

\begin{remark}
  \label{rem:montonicityBGB}
  Our notion of path
  monotonicity is the standard one for the Lagrangian Grassmannian
  \cite{Arn_faa67} and is local in nature.  If $\cL(t)$ never
  intersects the vertical subspace $\cV = 0 \oplus \cK$, then this
  coincides with the partial ordering defined in
  \cite[Sec.~5.2]{BehHasDeS_boundarytriples} for families of
  self-adjoint linear relations.
  Namely, $\cL(\cdot)$ is non-decreasing in the sense of
  \Cref{def:increasing_path} if and only if $\cL(t_1) \leq \cL(t_2)$
  for all $t_1 \leq t_2$ in the sense of
  \cite[Def.~5.2.3]{BehHasDeS_boundarytriples}.
  
  This is no longer the case if $\cL(t)$ intersects $\cV$ at some
  time. A simple example is
  $\cL(t) = \{ (z \cos t, z \sin t) : z \in \bbC\}$, which is
  Lagrangian in $\bbC^2$ for real values of $t$. At an arbitrary point
  $v = (z_0 \cos t_0, z_0 \sin t_0) \in \cL(t_0)$ we compute
  $\form[v] = |z_0|^2$, and conclude that $\cL(\cdot)$ is increasing
  in the sense of \Cref{def:increasing_path}. On the other hand,
  for small positive $\epsilon$, 
  $\cL(\pi/2+\epsilon) \leq \cL(\pi/2-\epsilon)$ 
  in the sense of \cite[Def.~5.2.3]{BehHasDeS_boundarytriples}.
  \hfill$\Diamond$
\end{remark}


\section{The Duistermaat triple index}
\label{sec:index_formulas}

Again assuming that $\cK\oplus\cK$ is a finite-dimensional complex
symplectic space, we now recall the definition of the
\emph{Duistermaat index} $\iD$, which first appeared in
\cite[Eq.~(2.16)]{Dui_am76}.  It is closely related to other
symplectically invariant triple indices, such as Kashiwara--Wall index
\cite{CapLeeMil_cpam94,LionVergne_WeilMaslov,Wal_im69} --- also known
as the ``triple signature''
\cite[App.~6.2]{LibermannMarle_symplectic_geometry} --- and
Leray--de~Gosson index of inertia defined in
\cite[Def.~148]{deGosson_PrinciplesNewtonianQuantum} (generalizing
\cite[I.2.4]{Leray_LagrangianAnalysis} to remove the assumption of
transversality).  In fact there is only one non-trivial symplectic
invariant of a triple of Lagrangian planes
\cite[Prop.~4.4]{AgrGam_SymplMeth}, and the above indices are just
different incarnations of it.

When describing the Duistermaat index, we follow the original
definition \cite[Eq.~(2.16)]{Dui_am76}, but use the notational
conventions of \cite{ZhoWuZhu_fmc18}, whose results we will use
below.  An alternative axiomatic approach to the Duistermaat
index 
is presented in our follow-up work
\cite{BerCoxLatSuk_prep24}.

\subsection{Definition and basic properties}
For Lagrangian planes $\alpha$, $\beta$ and $\gamma$ such
that $\alpha \cap \beta = 0 = \beta \cap \gamma$, we can view $\gamma$ as
the graph of a linear mapping $L \colon \alpha \to \beta$, i.e., $\gamma=\{u+Lu: u\in\alpha\}$.  This
gives rise to a bilinear form $Q(\alpha, \beta; \gamma)\colon
\alpha \times \alpha \to \C$
acting by
\begin{equation}
  \label{eq:Qdef}
  Q(\alpha, \beta; \gamma)\colon
  (u_1, u_2) \mapsto \omega(u_1, L u_2).
\end{equation}
It is easily shown that
\begin{equation}
\label{eq:Qkernel}
	\ker Q(\alpha, \beta; \gamma) = \alpha \cap \gamma.
\end{equation}
For arbitrary $\cL_1, \cL_2, \cL_3 \in \Lagr$ we choose
an $\hat\cL$ that is transversal to all three and define 
the \emph{Duistermaat index} of $(\cL_1,\cL_2,\cL_3)$ to
be\footnote{The negative index of a bilinear form is defined as the
  maximal dimension of a subspace on which the form is negative
  definite.  Alternatively, one may count the number of negative
  eigenvalues of the corresponding Hermitian matrix, as in
  \eqref{eq:n_plus_minus_def}.} 
\begin{equation}
  \label{eq:Dui_triple_def}
  \iD(\cL_1,\cL_2,\cL_3)
  := n_-\big(Q(\cL_2, \hat\cL; \cL_3)\big)
    - n_-\big(Q(\cL_1, \hat\cL;  \cL_3)\big)
    + n_-\big(Q(\cL_1, \hat\cL; \cL_2)\big).
\end{equation}

One can geometrically describe the integer $\iD(\cL_1,\cL_2,\cL_3)$ as
the maximal dimension of a subspace $\widehat{\cL}_3 \subset \cL_3$
which lies between $\cL_1$ and $\cL_2$, where ``between'' is defined
in terms of the positive direction of rotation in the Lagrangian
Grassmannian,
as introduced in \Cref{sec:monotonicity}.
This geometric interpretation, which is not immediately obvious from
the definition, is illustrated by the following example (see also
\Cref{cor:DNmap}).

\begin{figure}
  \centering
  \includegraphics[scale=0.9]{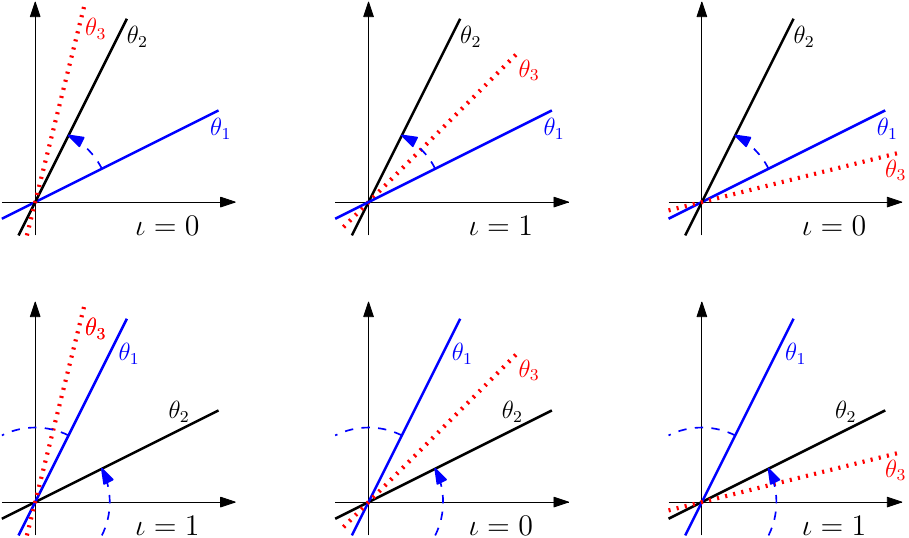}
  \caption{Possible transversal configurations of three Lagrangian
    planes (illustrated in $\R^2$) 
    and the corresponding values of the
    Duistermaat index. }
  \label{fig:iD_examples}
\end{figure}

\begin{example}
  \label{ex:basic_iD}
  In the symplectic space $\C^2$ consider the Lagrangian planes
  $\cL_j = \{ (z, \theta_j z) \colon z\in\C\}$ with $\theta_j \in \R$
  and $j=1,2,3$.  Simple calculations (for example, using
  \Cref{prop:Dui_index} where we can take $\epsilon=0$)
  show that
  \begin{equation}
    \label{eq:basic_iD}
    \iD(\cL_1,\cL_2,\cL_3)
    =
    \begin{cases}
      0,\quad & \theta_1 \leq \theta_2 \leq \theta_3
      \ \text{ or }\ 
      \theta_3 < \theta_1 \leq \theta_2
      \ \text{ or }\ 
      \theta_2 \leq \theta_3 < \theta_1,\\
      1,\quad & \theta_1 \leq \theta_3 < \theta_2
      \ \text{ or }\
      \theta_2 < \theta_1 \leq \theta_3
      \ \text{ or }\
      \theta_3 < \theta_2 < \theta_1.
    \end{cases}
  \end{equation}
  These results are illustrated in Figure~\ref{fig:iD_examples}.
\end{example}

We next recall some basic properties of $\iD$ that will be useful
later.  The Duistermaat index $\iD$ is a symplectic invariant: for any
symplectic automorphism $g$ of $(\cK\oplus\cK, \omega)$, we have
\begin{equation}
  \label{eq:symplectic_invariant}
  \iD(\cL_1,\cL_2,\cL_3) = \iD\!\big(g(\cL_1),g(\cL_2),g(\cL_3)\big).
\end{equation}
For any Lagrangian $\cL_1$, $\cL_2$, $\cL_3$ and $\cL_4$, $\iD$
satisfies the \emph{cocycle property}:
\begin{equation}
  \label{eq:cocyle_identity}
  \iD(\cL_1,\cL_2,\cL_3) - \iD(\cL_1,\cL_2,\cL_4)
  + \iD(\cL_1,\cL_3,\cL_4) - \iD(\cL_2,\cL_3,\cL_4) = 0.
\end{equation}
This follows from \cite[Thm~1.1]{ZhoWuZhu_fmc18}, cf.\
\cite[Prop.~1.5.8]{LionVergne_WeilMaslov},
\cite[Eq.~(I.2.13)]{deG_jmpa92}, \cite[Sec.~8,
Prop.~VI]{CapLeeMil_cpam94}.  In fact, the definition in
\eqref{eq:Dui_triple_def} can be interpreted as letting
\begin{equation}
  \label{eq:n_minus_Q}
  \iD(\cL_1,\cL_2,\cL_3) = n_-\big(Q(\cL_1,\cL_3;\cL_2)\big),
  \qquad \text{assuming }\cL_1 \cap \cL_3 = 0 = \cL_2\cap\cL_3,
\end{equation}
then using the cocycle property to extend to Lagrangian planes
with no transverality assumptions.  Equation~\eqref{eq:n_minus_Q} 
will be derived from the definition
\eqref{eq:Dui_triple_def} of $\iD$ in Corollary~\ref{cor:n_minus_Q}
below (see also \cite[Lem.~2.4]{Dui_am76} and
\cite[Lem.~3.13]{ZhoWuZhu_fmc18}).

From \cite[Lem.~2.4]{Dui_am76} we immediately get an estimate
\begin{equation}
  \label{eq:Dui_estimate}
  0 \leq \iD(\cL_1,\cL_2,\cL_3)
  \leq n - \dim\big((\cL_1\cap\cL_2)+(\cL_2\cap\cL_3)\big)
  \leq n - \dim \cL_1 \cap \cL_2.
\end{equation}
Under permutation of the first two arguments we have
\begin{equation}
  \label{eq:iDswap1}
  \iD(\cL_1,\cL_2,\cL_3) + \iD(\cL_2,\cL_1,\cL_3)
    = n - \dim \cL_1 \cap \cL_2.
\end{equation}
This follows, for instance, from
\Cref{prop:Dui_index} below. We also have the cyclic identity
\begin{align}
  \label{eq:iDshift1}
  \iD(\cL_1,\cL_2,\cL_3) - \dim\cL_1\cap\cL_3
    = \iD(\cL_3,\cL_1,\cL_2) - \dim\cL_2\cap\cL_3,
\end{align}
which follows from \cite[Lem.~3.2 and Lem.~3.13]{ZhoWuZhu_fmc18}.
Combining \eqref{eq:iDswap1} and \eqref{eq:iDshift1}, we obtain
identities for other permutations,
\begin{align}
  \label{eq:swap23}
    &\iD(\cL_1,\cL_2,\cL_3) + \iD(\cL_1,\cL_3,\cL_2)
      = n - \dim \cL_2 \cap \cL_3, \\
  \label{eq:swap13}
    &\iD(\cL_1,\cL_2,\cL_3) + \iD(\cL_3,\cL_2,\cL_1)
      = n - \dim \cL_1 \cap \cL_2 - \dim \cL_2 \cap \cL_3
      + \dim\cL_1\cap\cL_3, 
\end{align}
and the important special cases
\begin{equation}
  \label{eq:iD_special_cases}
  \iD(\cL, \cL, \cL_3) = 0,
  \qquad
  \iD(\cL_1, \cL, \cL) = 0,
  \qquad
  \iD(\cL, \cL_2, \cL) = n - \dim\cL\cap\cL_2.
\end{equation}

We finally recall an identity of Zhou, Wu and 
Zhu that relates the difference of
Maslov indices, with different reference planes, to the difference 
of Duistermaat indices:
\begin{equation}
    \Mas_{[a,b]}\!\big(\cL_2, \cM(\cdot)\big) -
    \Mas_{[a,b]}\!\big(\cL_1, \cM(\cdot)\big)
  =
    \iD\!\big(\cL_1, \cL_2, \cM(b) \big) - \iD\!\big(\cL_1, \cL_2, \cM(a) \big)
 \label{eq:Hor_ZWZ0}
\end{equation}
for any continuous path $\cM(\cdot)$; see \cite[Thm~1.1]{ZhoWuZhu_fmc18}. 
The difference of Maslov indices is also known as the \emph{H\"ormander index} \cite{Hor_am71}.

\subsection{One-sided limits}

The notion of monotonicity in \Cref{def:increasing_path} allows us to compute one-sided limits of the Duistermaat index.

\begin{theorem}
  \label{prop:continuous}
  Suppose $\cL\colon (-1,1) \to \Lagr$ is a
  continuous path that is differentiable and increasing on
  $(-1,0) \cup (0,1)$, and set $\cL_0 := \cL(0)$. For any $\cL_1, \cL_2, \cL_3 \in \Lagr$ and $0 < |t| \ll 1$ we have
  \begin{align}
    \label{eq:limit1}
    \iD\!\big(\cL(t), \cL_2, \cL_3 \big)
    &= \iD(\cL_0, \cL_2, \cL_3) +
      \begin{cases}
        0, & t<0, \\
        \dim \cL_2 \cap \cL_0 - \dim \cL_3 \cap \cL_0,
        & t >0,
      \end{cases} \\
    \label{eq:limit1alt}
    &\qquad=
      \begin{cases}
        \iD (\cL_0, \cL_2, \cL_3 ), & t<0,\\
        \iD (\cL_2, \cL_3, \cL_0 ), & t>0,
      \end{cases} \\
    \label{eq:limit2}
    \iD\!\big(\cL_1, \cL(t), \cL_3\big)
    &= \iD(\cL_1, \cL_0, \cL_3) +
      \begin{cases}
        \dim \cL_1 \cap \cL_0, & t < 0, \\
        \dim \cL_3 \cap \cL_0, & t > 0,
      \end{cases} \\
    \label{eq:limit3}
    \iD\!\big(\cL_1, \cL_2, \cL(t)\big)
    &= \iD(\cL_1, \cL_2, \cL_0) +
      \begin{cases}
        \dim \cL_2 \cap \cL_0 - \dim \cL_1 \cap \cL_0, & t < 0, \\
        0, & t > 0,
      \end{cases} \\ \label{eq:limit3alt}
    &\qquad=
      \begin{cases}
        \iD(\cL_0, \cL_1, \cL_2), & t<0,\\
        \iD(\cL_1, \cL_2, \cL_0), & t>0.
      \end{cases}
  \end{align}
\end{theorem}

In particular, we see that $\iD(\cL_1, \cL_2, \cL_3)$ is left-continuous in $\cL_1$ and right-continuous in $\cL_3$, but in general is neither right- nor left-continuous in $\cL_2$.

\begin{proof}
We first prove \eqref{eq:limit3}, as it follows most directly from the definition of $\iD$. Choose $\hat\cL$ that is transversal to $\cL_1$, $\cL_2$ and $\cL_0$ (and hence to $\cL(t)$ for small $t$). Recalling the definition in \eqref{eq:Dui_triple_def}, we have
\begin{equation}
  \label{eq:Dui_triple_def2}
  \iD\!\big(\cL_1,\cL_2,\cL(t) \big)
  = n_-\big(Q(\cL_2, \hat\cL; \cL(t))\big)
    - n_-\big(Q(\cL_1, \hat\cL;  \cL(t))\big)
    + n_-\big(Q(\cL_1, \hat\cL; \cL_2)\big).
\end{equation}
Starting with the first term on the right-hand side, we abbreviate
$Q_t := Q(\cL_2, \hat\cL; \cL(t))$. This acts by
$Q_t[u] = \omega(u, L_t u)$, where $L_t \colon \cL_2 \to \hat\cL$ is
such that $\cL(t) = \{u+L_tu \colon u \in \cL_2\}$. From
\Cref{THMform}, namely \eqref{Q-St} with $\cL^\sharp = \cL_2$, we see
that $Q_t'[u] = \form[u + L_t u]$ for any $u \in \cL_2$, therefore
$Q_t'$ is positive for $t \in (-1,0) \cup (0,1)$. It follows from the
mean value theorem that $Q_t$ is increasing on $(-1,1)$, therefore
$n_-(Q_t) = n_-(Q_0)$ for small $t>0$, and
\begin{equation}
  n_-\big(Q_t) = n_-(Q_0) + \dim \ker Q_0 
  =  n_-(Q_0) + \dim \cL_2 \cap \cL_0
\end{equation}
for small $t<0$, where we have used \eqref{eq:Qkernel}. An analogous formula holds for $n_-\big(Q(\cL_1, \hat\cL; \cL(t))\big)$. Using this in \eqref{eq:Dui_triple_def2} completes the proof of \eqref{eq:limit3}.

To prove \eqref{eq:limit1} we combine \eqref{eq:limit3} with the
identities
\begin{align}
  \label{eq:iD_swap0}
  &\iD(\cL_2, \cL_3, \cL_0)
    = \iD(\cL_0, \cL_2, \cL_3) + \dim \cL_2 \cap \cL_0 - \dim \cL_3
    \cap \cL_0,
  \\
    \label{eq:iDt_swap1}
  &\iD\!\big(\cL_2, \cL_3, \cL(t) \big)
    = \iD\!\big(\cL(t), \cL_2, \cL_3 \big),
    \qquad
    0 < |t| \ll 1,
\end{align}
which follow from \eqref{eq:iDshift1} and the observation
that $\cL(t)$ is transversal to $\cL_2$ and $\cL_3$ except at isolated
values of $t$.  The proof of \eqref{eq:limit2} is analogous.
\end{proof}

\begin{corollary}
  \label{cor:n_minus_Q}
  Under the assumption $\cL_1 \cap \cL_3 = 0 = \cL_2\cap\cL_3$ we have
  \begin{equation}
    \label{eq:n_minus_Q_copy}
    \iD(\cL_1,\cL_2,\cL_3) = n_-\big(Q(\cL_1,\cL_3;\cL_2)\big).
  \end{equation}
\end{corollary}

\begin{proof}
  Let $\cL(t)$ be an increasing path with $\cL(0) = \cL_3$. Since $\cL_3$
  is transversal to $\cL_1$ and $\cL_2$,  \Cref{prop:continuous} implies
  $\iD(\cL_1, \cL_2, \cL_3) = \iD(\cL_1, \cL_2, \cL(t))$ for
  $|t| \ll 1$.  For small nonzero $t$ we can choose $\hat\cL = \cL_3$
  in the definition \eqref{eq:Dui_triple_def} of $\iD(\cL_1, \cL_2, \cL(t))$ to get
  \begin{equation}
    \iD\!\big(\cL_1, \cL_2, \cL(t)\big)
    = n_-\big(Q(\cL_2,\cL_3; \cL(t))\big)
    - n_-\big(Q(\cL_1,\cL_3; \cL(t))\big)
    + n_-\big(Q(\cL_1,\cL_3; \cL_2) \big).
  \end{equation}
  Using the identity
  $n_-\big(Q(\alpha, \beta; \gamma)\big) + n_-\big(Q(\beta, \alpha;
  \gamma)\big) = n$, valid when $\alpha,\beta$ and $\gamma$ are pairwise 
  transversal,
  we can
  rewrite this as
  \begin{equation}
    \iD\!\big(\cL_1, \cL_2, \cL(t)\big)
    = n_-\big(Q(\cL_3,\cL_1; \cL(t))\big)
    - n_-\big(Q(\cL_3,\cL_2; \cL(t))\big)
    + n_-\big(Q(\cL_1,\cL_3; \cL_2) \big).
  \end{equation}
  The quadratic form $Q(\cL_3,\cL_1; \cL(t))$ is increasing in $t$ and
  is identically zero when $t=0$, thus
  $n_-\big(Q(\cL_3,\cL_1; \cL(t))\big) = 0$ for small positive $t$ and
  similarly for $n_-\big(Q(\cL_3,\cL_2; \cL(t))\big)$.
\end{proof}

\subsection{Computing with frames}
\label{sec:frames}

We now give a simple formula for computing the
Duistermaat index using linear algebra. First, we recall that any $n$-dimensional subspace 
$ \cM \subset \cK \oplus \cK$ can be described by a \term{frame}, which is 
an injective linear operator
\begin{equation}
  \label{eq:frame}
	Z = \Lframe{X}{Y} \colon \cK \to \cK\oplus\cK,
\end{equation}
whose range is $\cM$. Moreover, $\cM$ is Lagrangian if and only if 
$X^*Y = Y^*X$ (see Appendix~\ref{sec:param} for a review of this and
other parametrizations of Lagrangian planes).

This description is not unique, but it is easy to see that frames $Z$ and $\tilde Z$ describe the 
same subspace if and only if $\tilde Z = ZC$ for some invertible 
$C \colon \cK \to \cK$. Therefore, the set
\begin{equation}
\label{Edef}
	E(\cM) := \{ \epsilon \in \bbR : X + \epsilon Y \text{ is not invertible} \}
\end{equation}
and the operator
\begin{equation}
  \label{eq:almostDTN}
  R^\epsilon := Y(X + \epsilon Y) ^{-1},
  \qquad
  \epsilon \in \R \backslash E(\cM),
\end{equation}
are independent of the choice of frame. The set $E(\cM)$ is finite,
since $\det(X + \epsilon Y)$ is a polynomial in $\epsilon$ that is not
identically zero because $Z$ has rank $n$. When $\cM$ is Lagrangian,
the condition $X^*Y = Y^*X$ implies that
\begin{equation}
  \label{eq:HermitianR}
  (X+\epsilon Y)^* R^\epsilon (X+\epsilon Y) = X^*Y + \epsilon Y^* Y
\end{equation}
is Hermitian, therefore $R^\epsilon$ is Hermitian.

\begin{remark}
  \label{rem:RM}
For the Cauchy data space $\cM(z)$ defined in \eqref{eq:CDS_def},  the 
corresponding $R^\epsilon(z)$ acts by $u \mapsto \Gamma_1 f$, where $f \in \ker(\minop^* - z)$ 
satisfies $\Gamma_0 f + \epsilon \Gamma_1 f = u$. (The condition
$\epsilon \in \R\setminus E\big(\cM(z)\big)$ 
guarantees there is a unique such $f$ for each $u \in \cK$.)   In particular, $R^0(z)$ is the Dirichlet-to-Neumann map, whenever it is is defined. We thus refer to the operator $R^\epsilon$ in \eqref{eq:almostDTN} as the \emph{$\epsilon$-Robin map}, whether or not the corresponding subspace $\cM$ is the Cauchy data space.
\hfill$\Diamond$\end{remark}

  An intuitive description of $R^\epsilon$ is the ``regularized
  slope'' of $\cM$, as drawn in $\cK \oplus \cK$.  Referring to \Cref{fig:iD_examples}, 
  it is therefore natural that the Duistermaat
  index can be computed by comparing slopes of pairs of planes, 
  in the following sense.

\begin{theorem}
  \label{prop:Dui_index}
  Let $R^\epsilon_1$, $R^\epsilon_2$ and $R^\epsilon_3$ be the $\epsilon$-Robin maps
  corresponding to the Lagrangian planes $\cL_1$, $\cL_2$ and $\cL_3$.  Then the Duistermaat index
  $\iD(\cL_1,\cL_2,\cL_3)$ is given by
  \begin{equation}
      \label{eq:Dui_index_gen}
      \iD(\cL_1,\cL_2,\cL_3) 
	= \ n_-\Big( R^\epsilon_2-R^\epsilon_1 \Big)
	+ n_-\Big( R^\epsilon_3-R^\epsilon_2 \Big)
	- n_-\Big( R^\epsilon_3-R^\epsilon_1 \Big)
  \end{equation}
  for any $\epsilon \in \R \backslash E_{123}$, where
  $E_{123} := E(\cL_1)\cup E(\cL_2)\cup E(\cL_3)$ is a finite set.
  In particular, for the vertical plane
  $\cV = 0 \oplus \cK$ we have
  \begin{equation}
    \label{eq:Dui_index_L3vertical}
    \iD(\cL_1,\cL_2,\cV) 
    = \ n_-\Big( R^\epsilon_2-R^\epsilon_1 \Big)
  \end{equation}
  for any $0 < \epsilon \ll 1$.
\end{theorem}

\begin{proof}
  We first assume that $\cL_1$, $\cL_2$ and $\cL_3$ are transversal to
  $\cV$ and establish \eqref{eq:Dui_index_gen} with $\epsilon=0$,
  using the definition \eqref{eq:Dui_triple_def} with $\hat\cL=\cV$.
  In this case each $\cL_j$ is represented by the frame $(I, R^0_j)$.
  We will use this to compute the index of $Q(\cL_2, \cV; \cL_3)$ to
  be used in \eqref{eq:Dui_triple_def}. To represent $\cL_3$ as the
  graph of an operator $L\colon \cL_2 \to \cV$, we write
  \begin{equation}
    \label{eq:x2_expand}
   	\cL_3 \ni \begin{pmatrix}
      \kappa \\ R^0_3 \kappa
    \end{pmatrix}
    = 
   \begin{pmatrix}
      \kappa \\ R^0_2 \kappa
    \end{pmatrix}
    +
   \begin{pmatrix}
      0 \\ R^0_3\kappa - R^0_2\kappa
    \end{pmatrix} = u + L u \in \cL_2 + \cV.
  \end{equation}
  For any $u_1 = (\kappa_1, R^0_2 \kappa_1)^T$ and
  $u_2 = (\kappa_2, R^0_2 \kappa_2)^T$ in $\cL_2$
we thus obtain
  \begin{equation}
    \label{eq:Q23}
    Q(\cL_2, \cV; \cL_3) \colon (u_1, u_2) \longmapsto
    \omega( u_1, L u_2)
    =
    \left< \kappa_1,  (R^0_3 - R^0_2) \kappa_2 \right>_{\cK}.
  \end{equation}
Since the index is invariant under isomorphism, the form 
  $Q(\cL_2, \cV; \cL_3)$ on $\cL_2$ has the same index as the form 
  $\left< \cdot ,  (R^0_3 - R^0_2) \cdot \right>_{\cK}$ on $\cK$, that is
  \begin{equation}
    \label{eq:index_Q23}
    n_-\big(Q(\cL_2, \cV; \cL_3)\big) = n_-\big(R^0_3 - R^0_2\big).
  \end{equation}
  Evaluating the other two terms in \eqref{eq:Dui_triple_def}
  similarly, we obtain \eqref{eq:Dui_index_gen} with $\epsilon=0$, i.e.,
  \begin{equation}
      \label{eq:Dui_index_gen0}
      \iD(\cL_1,\cL_2,\cL_3) 
	= \ n_-\Big( R^0_2-R^0_1 \Big)
	+ n_-\Big( R^0_3-R^0_2 \Big)
	- n_-\Big( R^0_3-R^0_1 \Big).
  \end{equation}
  
  In the general case, we can use the symplectic transformation
  \begin{equation}
    \label{eq:symplectic_small_rotation}
    g_\epsilon =
    \begin{pmatrix}
      I & \epsilon I \\
      0 & I
    \end{pmatrix}
  \end{equation}
  to make $\cL_1$, $\cL_2$ and $\cL_3$ transversal to
  $\cV$.  In other words, we consider Lagrangian planes
  \begin{equation}
    \label{eq:Leps}
    \cL^\epsilon_j := g^\epsilon (\cL_j) = \Ran
    \begin{pmatrix}
      X_j + \epsilon Y_j \\ Y_j
    \end{pmatrix},
  \end{equation}
  with $\epsilon$ chosen so that all $X_j+\epsilon Y_j$ are
  invertible. From \eqref{eq:symplectic_invariant} we have $\iD(\cL_1,\cL_2,\cL_3) =
  \iD(\cL^\epsilon_1,\cL^\epsilon_2,\cL^\epsilon_3)$, so we can compute
  $\iD(\cL^\epsilon_1,\cL^\epsilon_2,\cL^\epsilon_3)$ according to
  \eqref{eq:Dui_index_gen0} and thus obtain
  \eqref{eq:Dui_index_gen}.

  Finally, we consider the case $\cL_3 = \cV$ and prove \eqref{eq:Dui_index_L3vertical}.  We have
  $R^\epsilon_3 = \epsilon^{-1} I$, so the result follows once we
  establish that for any Lagrangian frame $(X, Y)^T$, the operator
  $\epsilon^{-1} I - Y(X+\epsilon Y)^{-1}$ is non-negative definite.
  It is equivalent to consider 
  \begin{equation}
    \label{eq:similarity}
    \epsilon (X+\epsilon Y)^*
    \left(\epsilon^{-1} I - Y(X+\epsilon Y)^{-1}\right)
    (X+\epsilon Y)
    =  X^* X+\epsilon Y^* X.
  \end{equation}
  The right-hand side is a Hermitian perturbation of
  the non-negative operator $X^*X$.  As functions of $\epsilon$, zero
  eigenvalues of the unperturbed problem remain identically zero,
  since the perturbation $Y^*X$ vanishes on $\ker(X^*X) = \ker X$.  On the
  other hand, the non-zero eigenvalues are positive at $\epsilon=0$
  and thus remain bounded away from zero for small $\epsilon$.
\end{proof}

Using \Cref{prop:Dui_index}, we obtain a formula for the
Duistermaat index in the case when $\cL_3 = \cV = 0\oplus\cK$ and
$\cL_1 \cap \cV = 0$.  We remark here that any Lagrangian plane
can be described in terms of a frame $(X,Y) = (P,\, P \Theta P + P - I)$, where
$P \colon \cK \to \cK$ is an orthogonal projector and $\Theta$ is a
Hermitian operator acting on $\Ran P$; see \Cref{sec:param}.

\begin{proposition}
  \label{prop:BL_index_formula}
  Suppose the planes $\cL_1$ and $\cL_2$ are
  described by the frames $(I,\, M)$ and $(P,\, P \Theta P + P - I )$,
respectively.  Then
  \begin{equation}
    \label{eq:BL_index_formula}
    \iD(\cL_1, \cL_2, \cV) = n_-(\Theta - P M P).
  \end{equation}
\end{proposition}

This result is inspired by a counting formula in \cite{BehLug_jpa10}
for the eigenvalues of the Laplacian on a metric graph, which we will
rederive in \Cref{sec:BL}.

\begin{proof}
  Let $\cL(t)$ denote the path given by the frames $(I, M+tI)$. Since
  $\cL(t)$ is increasing and $\cL(0) = \cL_1$, \Cref{prop:continuous}
  gives
  \begin{equation}
    \label{eq:cont_left}
    \iD(\cL_1,\cL_2,\cL_3) = \lim_{t\to0-} \iD\!\big(\cL(t),\cL_2,\cL_3\big).
  \end{equation}
  We now use equation~(\ref{eq:Dui_index_L3vertical}) to compute
  $\iD\!\big(\cL(t),\cL_2,\cL_3\big)$.  Writing
  $R_j^\epsilon = Y_j(X_j + \epsilon Y_j)^{-1}$ in block form corresponding
  to the decomposition $\cK = \ker P \oplus \Ran P$, we obtain
  \begin{equation}
    \label{eq:R12diff}
    R_2^\epsilon - R_1^\epsilon = 
    \begin{pmatrix}
      \epsilon^{-1} I - M_{11} - t I & -M_{12} \\
      -M_{21} & \Theta - M_{22} - t I
    \end{pmatrix} + O(\epsilon),
  \end{equation}
  where $M_{11} = (I-P)M(I-P)$, $M_{22} = PMP$ and so on.  The
  top-left block is strictly positive (and in particular invertible)
  for small $\epsilon>0$, so the Haynsworth formula \cite{Hay_laa68}
  implies
  \begin{align}
    n_-\left(R_2^\epsilon - R_1^\epsilon\right) 
    &=
      n_-\left(\Theta - M_{22} - t I + O(\epsilon)
      + M_{21} \left(\epsilon^{-1} I + O(1)\right)^{-1} M_{12} \right) \nonumber \\
    &= n_-\big(\Theta - M_{22} - t I + O(\epsilon) \big)
      \label{eq:Hainsworth}
  \end{align}
  as $\epsilon \to 0$. 
  The operator $\Theta - M_{22} -t I_2$ is invertible for all
  $t$ in some interval $(t_*, 0)$, therefore
  \begin{equation}
    \label{eq:eps_limit}
    \iD\!\big(\cL(t),\cL_2,\cL_3 \big) = \lim_{\epsilon \to 0+}
    n_-\left(R_2^\epsilon - R_1^\epsilon\right)
    = n_-\left(\Theta - M_{22} - t I_2\right)
    = n_-\left(\Theta - M_{22} + |t| I_2\right).
  \end{equation} 
  On the other hand, since negative eigenvalues cannot be produced
  by a small positive perturbation, we have, for sufficiently small $t$,
  \begin{equation}
    \label{eq:final_limit}
    \iD\!\big(\cL(t),\cL_2,\cL_3\big)
    = n_-\left(\Theta - M_{22}\right)
    = n_-\left(\Theta - PMP\right),
  \end{equation}
  completing the proof.
\end{proof}

Finally, we give a corollary that will be useful in applications, and also clearly 
illustrates the idea that the index $\iD(\cL_1,\cL_2,\cL_3)$ quantifies how much of $\cL_3$
 lies ``between" $\cL_1$ and $\cL_2$ (in the special case that $\cL_1$ and $\cL_2$ are horizontal and vertical, 
 respectively).

\begin{corollary}
  \label{cor:DNmap}
  Let $\widehat{\cK} \subseteq \cK$ be a subspace and 
  $\Theta \colon \widehat{\cK} \to \widehat{\cK}$ a self-adjoint operator.
  Denote its number of 
  non-negative eigenvalues by $n_{0+}(\Theta)$ and consider the
  Lagrangian plane
  $\cL_\Theta := \big\{ (\kappa, \kappa'+\Theta\kappa) \colon
    \kappa\in\widehat{\cK},\ \kappa'\in\widehat{\cK}^\perp\big\}$.  Then
  \begin{equation}
    \label{eq:DN_iD}
    \iD\!\big( \cK\oplus0,\, 0\oplus\cK,\, \cL_\Theta\big)
    = n_{0+}(\Theta).
  \end{equation}
\end{corollary}

\begin{remark}
  If we view $\cL_\Theta$ as a self-adjoint linear relation from $\cK$ to $\cK$,
  then $\Theta$ is the ``operator part'' of
  $\cL_\Theta$, as in \cite[Prop.\ 14.2]{Schmudgen_unboundedSAO}.
\end{remark}

\begin{proof}
  Using \Cref{prop:BL_index_formula} with $\cL_1 = \cK\oplus0$, $\cL_2
  = \cL_\Theta$, $M=0$ and $P$ being the orthogonal projector onto
  $\widehat{\cK}$, we get $\iD( \cK\oplus0, \cL_\Theta, 0\oplus\cK)
    = n_{-}(\Theta)$. Since
    \[
    	n - \dim\big( \cL_\Theta \cap (0\oplus\cK) \big) = n - \dim \widehat{\cK}^\perp = \dim \widehat\cK
	 = n_-(\Theta) + n_{0+}(\Theta),
    \]
	the result follows from \eqref{eq:swap23}.
\end{proof}

\section{The Cauchy data space}
\label{sec:CauchyDataSpace}

The main object in the proof of \Cref{thm:main,thm:shift_Hormander} is the Cauchy 
data space
\begin{equation}
  \label{eq:CDS_recall}
  \cM(z) = \tr \big(\!\ker(\minop^*-z) \big) 
  \subset \cK \oplus \cK
\end{equation}
introduced in \eqref{eq:CDS_def}, where $\tr =(\Gamma_0, \Gamma_1)$.
We now establish its fundamental properties, which will be needed
below, in particular in the proofs of
\Cref{lem:Maslov_counting,thm:resolvent_diff_new}.  In this section we
allow $\cK$ to be infinite dimensional, as the results presented
herein are of independent interest.

Recall that the \term{deficiency} of a closed operator $T \colon X\to Y$ is
the codimension of $\Ran T$ in $Y$, i.e., $\deff T := \dim (Y/\Ran T)$, see
\cite[Sec.~I.3]{EdmundsEvans_spectral}.  
We define 
\begin{equation}
  \label{eq:Phi_minus_def}
  \Phi_-(T) :=\left\{z\in\C \colon \deff (T-z) <\infty  \right\}
\end{equation}
and observe that $z \in \Phi_-(T)$ implies $\Ran(T-z)$ is closed, by \cite[Thm.~I.3.2]{EdmundsEvans_spectral}. 
If $\minop$ is closed and symmetric, then
\begin{equation}
  \label{eq:Phi_minus_ess}
  \C \backslash \spess(H) \subseteq \Phi_-(\minop^*)
\end{equation}
for any self-adjoint extension $H$ of $\minop$, with equality when the
defect numbers of $\minop$ are finite; see
\cite[Cor.~IX.4.2]{EdmundsEvans_spectral}.

We now show that $\cM(z)$ depends on $z$ analytically as long as
$z\in \Phi_-(\minop^*)$.

\begin{theorem}
  \label{thm:analyticM_general}
  Let $\minop$ be a closed, densely defined symmetric operator on a Hilbert space $\cH$
  with equal (possibly infinite) defect numbers, and let $(\cK,\Gamma_0,\Gamma_1)$ be a boundary triplet.
  In a neighborhood of any ${z_0}\in \Phi_-(\minop^*)$ there exists an analytic
  family of invertible operators $G_{z} \in \mathcal{B}(\cK\oplus\cK)$
  such that $G_{z_0}=I_{\cK \oplus \cK}$ and $G_{z} \cM({z_0}) = \cM(z)$.
\end{theorem}

An equivalent formulation of the theorem is that $\cM(z)$ is an \emph{analytic Banach bundle} over $\Phi_-(\minop^*)$; see \cite{ZaiKreKucPan_umn75} for definitions. To compare this to previous results in the literature, we first recall that $\minop$ has the \emph{unique continuation property} (or, equivalently, has no \emph{inner solutions})
if $\ker (\minop^*-z)\cap \ker(\Gamma_0)\cap\ker(\Gamma_1)=0$ 
for all $z \in \C$. Since $\minop\subset \minop^*$, the identity
\begin{equation}\lb{c.8}
	\ker(\tr)=\ker(\Gamma_0)\cap\ker(\Gamma_1)=\dom(\minop),
\end{equation}
see, e.g., \cite[Lem.~14.6 (iv)]{Schmudgen_unboundedSAO}, 
 implies that this is equivalent to
	\begin{equation}\lb{c12}
	\ker(\minop-z)= 0.
	\end{equation}
Nontrivial elements of $\ker(\minop-z)$, if they exist, are called \emph{inner solutions}.

%
  In general $\ker(\minop^*-z)$ is not analytic or even continuous
  on $\Phi_-(\minop^*)$, since the dimension of $\ker(\minop^*-z)$ will
  jump at points $z$ which are eigenvalues of $\minop$.  
   For this reason, similar 
  results in the literature have assumed one of the following:
  \begin{enumerate}
  	\item $\minop$ has no inner solutions; see \cite[Thm.~3.8]{BooFur_tjm98} and \cite[\S6]{Fur_jgp04}.
	\item $z \in \rho(H_0)$, where $H_0$ is the Dirichlet-type extension of
  $\minop$; see \cite[Thm.~5.5.1]{BehHasDeS_boundarytriples} and \cite[Prop.~14.15]{Schmudgen_unboundedSAO}. 
  Note that this is stronger than assuming $\ker(\minop - z) = 0$.
  \end{enumerate} 
  \Cref{thm:analyticM_general}, on the other hand, requires no such assumptions on $\minop$ or $z$. 
The reason is that when passing to the Cauchy data $\cM(z)$, the jump in the dimension  disappears, leaving
  only the ``analytic component'' of $\ker(\minop^*-z)$.
  To make this
  intuition rigorous, we use the observation of M.~Krein
  \cite{Kre_umz49} (see also
  \cite[Rem.~2.3.10]{BehHasDeS_boundarytriples}) that one can split
  off a maximal self-adjoint part of $\minop$\,---\,which is responsible for
  the inner solutions\,---\,and hence consider only simple symmetric
  operators. 
  
  We recall from \cite[\S3.4]{BehHasDeS_boundarytriples}
that a closed, symmetric operator is \emph{simple} if $0$ is the largest reducing subspace 
on which it is self-adjoint. Simple symmetric operators have no 
eigenvalues, by \cite[Lem.~3.4.7]{BehHasDeS_boundarytriples}, and hence satisfy \eqref{c12} for all $z\in\C$.  Since $\minop$ is closed and symmetric, there is a 
splitting $\cH=\cH_{sim}\oplus \cH_{sa}$ with respect to which $\minop$ is diagonal and
  \begin{equation}
  \label{symdecomp}
  	\minop_{sim}:=\minop\big|_{\cH_{sim}}, \qquad \minop_{sa}:=\minop\big|_{\cH_{sa}}
  \end{equation}
  are simple symmetric and self-adjoint in $\cH_{sim}$ and $\cH_{sa}$, respectively; see \cite[\S3.4]{BehHasDeS_boundarytriples} for details.

\begin{proof}[Proof of \Cref{thm:analyticM_general}]
  \emph{Step one: reducing to the simple symmetric case.}
  Decomposing $\cH$ as in \eqref{symdecomp}, we have $\minop=\minop_{sim}\oplus \minop_{sa}$ and hence
  \begin{equation}
    \label{eq:decomposition_simple_symm}
    \minop^*-z=(\minop^*_{sim}-z)\oplus (\minop_{sa}-z)
  \end{equation}
  for all $z \in \C$. Since
  $\dom(\minop_{sa})\subset \dom(\minop) =\ker\tr$, we
  have $\tr \big(\!\ker(\minop_{sa}-z) \big) = 0$ and thus
  \begin{equation}
    \label{eq:cMsimple}
    \cM(z) = \tr \big(\!\ker(\minop^*-z) \big)
    = \tr \big(\!\ker(\minop^*_{sim}-z)\big).
  \end{equation}
  It therefore suffices to prove the result for $\minop_{sim}$, so we 
  will assume for the rest of the proof that $\minop$ is simple.
 
  \emph{Step two: $\minop^* - z$ is onto for $z\in\Phi_-(\minop^*)$.}
  By the definition of $\Phi_-$, the range of $\minop^*-z$ has finite
  codimension and hence is closed.  On the other hand, $\minop$ being
  simple implies $\ker(\minop-z)=0$, therefore $\Ran(\minop^*-z)$ is
  dense.
  
  \emph{Step three: $\ker(\minop^*-z)$ is analytic in $z$ when $\minop^*-z$ is onto.}
  Let $\cH_+:=\dom(\minop^*)$, equipped with the graph scalar product of
  $\minop^*$, so $\cH_+$ is a Hilbert space and
  $\minop^*\in\mathcal{B}(\cH_+, \cH)$. Since $\minop^*-z$ is surjective 
and its kernel (a closed subspace of a Hilbert space) is complemented, it
  has a bounded right inverse \cite[Thm.~2.12]{Brezis_FA}, i.e.,
  $B_z \in \cB(\cH,\cH_+)$ such that $(\minop^*-z) B_z = I_{\cH}$.  In a
  neighborhood of any $z_0\in \Phi_-(\minop^*)$, $B_z$ can be chosen to be
  analytic using the formula \[B_z := B_{z_0} \big( (\minop^*-z) B_{z_0}\big)^{-1} =
  B_{z_0} \big( I_{\cH} + ({z_0}-z) B_{z_0}\big)^{-1}.\]  
  Now
  \begin{equation}
    \label{eq:Pz_def}
    P_z \colon \cH_+ \to \cH_+,
    \qquad
    P_z := I_{\cH_+} - B_z (\minop^*-z),
  \end{equation}
  defines an analytic family of projectors (in
  general not orthogonal) onto $\cN_z = \ker(\minop^*-z)$.
  
  \emph{Step four: transformation functions (see
  \cite[\S II.4.2]{Kato_perturbation},
  \cite[\S IV.1.1]{DaletskiiKrein}) for $\cN_z$.}
  Fixing ${z_0}\in\Phi_-(\minop^*)$, define the operator family $F_z\colon \cH_+ \to \cH_+$ by
  \begin{equation}
    \label{eq:Fs_definition}
    F_z := (I - P_z)(I- P_{z_0}) + P_z P_{z_0} = I + (2 P_z - I)(P_{z_0}-P_z).
  \end{equation}
  The latter expression and $F_{z_0}=I$ show that $F_z$ is invertible for $z$ close
  to ${z_0}$; from the former expression we immediately get $F_zP_{z_0} = P_z
  F_z$ and $F_z^{-1}P_z = P_{z_0} F_z^{-1}$, therefore
  \begin{equation}
    \label{eq:Fs_property}
    F_z \cN_{z_0} = \cN_z.
  \end{equation}

  \emph{Step five: a right inverse for the boundary trace.} The trace
  operator $\tr \colon u\mapsto(\Gamma_0u,\Gamma_1u)$ is a bounded surjection
  from $\cH_+$ onto $\cK\oplus\cK$ with $\ker\tr=\dom(\minop)$, as in  
  \eqref{c.8}, so it has a bounded right inverse $\tr^R$, which can be
  chosen to satisfy
  \begin{equation}
    \label{eq:trR_condition}
    \tr^R\, \tr \big|_{\cN_{z_0}} = I_{\cN_{z_0}}.
  \end{equation}
  More explicitly, since $\cN_{z_0} \cap \ker\tr = 0$, the operator
  $\tr \colon \left(\dom(\minop) \oplus \cN_{z_0}\right)^\perp \oplus
  \cN_{z_0} \to \cK\oplus\cK$ is an isomorphism and $\tr^R$ is the
  corresponding inverse.
  
  \emph{Step six: transformation functions for $\cM(z)$.}
  Finally, the analytic family of operators
  \begin{equation}
    \label{eq:Gz_definition}
    G_{z} \colon \cK\oplus\cK \to\cK\oplus\cK,
    \qquad
    G_{z} := \tr F_z \tr^R
  \end{equation}
  satisfies
  \begin{equation}
    \label{eq:GMz_computation}
    G_{z} \cM({z_0}) = \tr F_z \tr^R \cM({z_0}) = \tr F_z \tr^R\, \tr \cN_{z_0}
    = \tr F_z \cN_{z_0} = \tr \cN_z = \cM(z)
  \end{equation}
  and is invertible for $z$ close to ${z_0}$ because $G_{{z_0}} = \tr \tr^R = I_{\cK\oplus\cK}$.
\end{proof}

\begin{remark}
  \label{rem:smooth_projector_family}
  From $G_z$
  one can define an analytic family of oblique projectors onto $\cM(z)$ by
  \[P_z=\begin{pmatrix}I&0\\G_z^{21}(G_z^{11})^{-1}&0\end{pmatrix},\quad
  \text{ where } \ G_z=\begin{pmatrix} G_z^{11}&G_z^{12}\\G_z^{21}&G_z^{22}\end{pmatrix} \]
  is the block decomposition of the operator $G_z$ in the direct sum decomposition $\cM(z_0)\oplus\cM(z_0)^\bot$. 
 Using \cite[Lem.~12.8]{BW93}, we see that the corresponding family of \emph{orthogonal} projections, 
$P_z P_z^* ( P_z P_z^* + (I-P_z^*)(I - P_z) )^{-1}$, 
is smooth. It is not analytic, however, since a family of orthogonal projections is analytic only if it is constant.
\hfill$\Diamond$\end{remark}

\newcommand{\uh}{\hat{u}}

The existence of the family $G_z$ allows us to define the crossing form
as in \eqref{Q;Gt} and thereby extend
the notion of monotonicity, \Cref{def:increasing_path}, to our present
setting of possibly infinite dimensional $\cK$.  In fact, the crossing form for
$\cM(z)$ has a beautiful explicit form.

\begin{corollary}
  \label{cor:Mdot}
  Fix $z_0 \in \Phi_-(\minop^*)$ and 
  let $G_z$ be the operator family from \Cref{thm:analyticM_general}.   
  Then, for any $v\in \cM(z_0)$,
  \begin{align}
  \begin{split}
    \form[v] := \frac{d}{dz} \omega(v, G_zv) \Big|_{z=z_0}
    &= \min\big\{
      \|g\|^2_\cH : (\minop^*-z_0)g=0,\ \tr g = v\big\}  \\
    \label{eq:form_Mdot}
    &= \|f\|_{\cH_{sim}}^2,
    \end{split}
  \end{align}
  where $\omega$ is the symplectic form \eqref{eq:sympl_form},
  $\cH_{sim}$ and $\minop_{sim}$ are defined in \eqref{symdecomp} and
  $f$ is the unique vector in $\cH_{sim}$ with
  $\minop_{sim}^*f = z_0 f$ and $\tr f = v$.
\end{corollary}

\begin{remark}
  Equation \eqref{eq:form_Mdot} generalizes known formulas:
  for the derivative of the Dirichlet-to-Neumann map 
  in the resolvent set of the ``Dirichlet'' extension
  \cite[Prop.~14.15(iv)]{Schmudgen_unboundedSAO}; and for the crossing form when $\minop$ satisfies the unique continuation condition, for instance 
  \cite[Thm.~5.1]{BooFur_tjm98} or \cite[Thm.~5.10]{LS1}.  Under such conditions, the
  solution $f$ to $\minop^* f = z_0 f$, $\tr f = v$ is unique and
  the operator that maps the first component of the vector $v\in\cM(z)\subset\cK\oplus\cK$ into $f$ is known as the \term{$\gamma$-field}.
\hfill$\Diamond$\end{remark}

\begin{proof}[Proof of \Cref{cor:Mdot}] 
  In view of the decomposition~\eqref{eq:decomposition_simple_symm} and
  its properties, the general solution of $\minop^* g = z_0 g$, $\tr g =
  v$ has the form
  \begin{equation}
    \label{eq:gamma_field_gen}
    g = f + \ker(\minop_{sa}-z_0),
  \end{equation}
  with $\|g\|^2 \geq \|f\|^2$, where $f\in\ker(\minop^*_{sim}-z_0)\subset\cH_{sim}$ is 
  unique by \eqref{c12}. Therefore, we can restrict ourselves to the
  case when $\minop$ is simple symmetric. In this case, 
  there exists a unique $f \in \ker(\minop^*-z_0)$ with $v = \tr f$.
  From \eqref{eq:Fs_property}, \eqref{eq:trR_condition} and
  \eqref{eq:Gz_definition} we have
  \begin{equation}
    \label{eq:Gzv}
    G_zv = \tr F_z \tr^R v = \tr F_z \tr^R \tr f = \tr F_z f
    = \tr f_z,
  \end{equation}
  where $f_z := F_z f \in \ker(\minop^*-z)$.  By Green's identity~\eqref{eq:Greens_identity}, we get
  \begin{align}
    \nonumber
    \omega(v, G_zv) = \omega(\tr f, \tr f_z)
    &= \left<f, (\minop^*-z_0) f_z\right>_\cH - \left<(\minop^*-z_0)f, f_z\right>_\cH \\
    \label{eq:Greens_form}
    &= \left<f, (z-z_0) f_z \right>_\cH = \left<f, (z-z_0) F_z f \right>_\cH.
  \end{align}
  Equation~\eqref{eq:form_Mdot} follows from $F_{z_0}=I$ and the continuity of $F_z$.
\end{proof}

We now state some further properties of the Cauchy data space, 
recalling $\cF=\tr\big(\!\dom (H_F) \big)$ for the Friedrichs extension $H_F$ of $S$.

\begin{proposition}
  \label{prop:cM_real_s} Under the assumptions in \Cref{thm:analyticM_general}
  the Cauchy data space $\cM(\cdot)$ has the following properties:
  \begin{enumerate}
  \item \label{item:complement} For all $z\in \Phi_-(\minop^*)$,
    \begin{equation}
      \label{eq:sympl_complement}
      \cM(\zbar) = \cM(z)^\omega
      \ :=\  \big\{u \in \cK \oplus \cK \colon \omega(u, v) = 0\
        \text{for all } v \in \cM(z) \big\}.
    \end{equation}
    In particular, $\cM(s)$ is Lagrangian and increasing for
    $s\in \Phi_-(\minop^*) \cap \R$.
  \item \label{item:inf_limit} If $\minop$ is bounded from below, with lower bound $\gamma$,
     then $\cM(s)$ has limits (over real $s$)
    \begin{equation}
      \label{eq:Mlimit}
      \lim_{s\to-\infty} \cM(s) = \cF, \qquad 
      \lim_{s\to \gamma-} \cM(s) =: \cM(\gamma-),
    \end{equation}
    in the strong graph sense, and the limiting subspace $\cM(\gamma-)$
    is Lagrangian.
  \end{enumerate}
\end{proposition}

We recall from \cite[Def.~1.9.1]{BehHasDeS_boundarytriples} that the strong
graph limit of $\cM(s)$ consists of all $u \in \cK\oplus\cK$ for which
there exists a sequence $u_s \in \cM(s)$ with $u_s \to u$.

\begin{remark}
  \label{rem:cM0}
  In general $\cM(\gamma-)\neq \cM(\gamma)$. For the example of $S = -d^2/dx^2$ 
  on the half-line (with the standard Dirichlet and Neumann traces), we have 
  $\gamma=0\in\spess(S)=\C\setminus \Phi_-(\minop^*)$ and
  $\cM(0)$, when computed from the definition \eqref{eq:CDS_recall},
  is equal to the zero subspace of $\C\oplus\C$ (and, in particular,
is not Lagrangian).  On the other hand, the
  limit $\cM(0-) $ is the Lagrangian plane $\C \oplus 0$. 
  \hfill$\Diamond$
\end{remark}

\begin{proof}[Proof of \Cref{prop:cM_real_s}]
  \eqref{item:complement}\,
  It follows from Green's identity~\eqref{eq:Greens_identity} that $\omega(u,v)=0$ 
  for all $u\in\cM(\zbar)$ and $v\in\cM(z)$,
  therefore $ \cM(\zbar) \subset \cM(z)^\omega$. To prove the other inclusion, 
  suppose $u\in \cM(z)^\omega$, so $\omega( u, \tr f) = 0$ for all $f\in\ker(\minop^*-z)$.
  %
  Since $\tr$ is surjective, there exists $g\in\dom(\minop^*)$ such that
  $\tr g= u$. Using Green's identity and $(\minop^*-z)f=0$, we get
  \begin{align}
  \begin{split}
    0 
    = \omega\big(\tr g, \tr f\big)
    &= \langle \minop^* g, f\rangle_{\cH}
    -\langle g, \minop^* f\rangle_{\cH} \\
    &= \big\langle (\minop^*-\zbar)g, f\big\rangle_{\cH}
        -\big\langle g, (\minop^*-z)f\big\rangle_{\cH}
    = \big\langle (\minop^*-\zbar)g, f\big\rangle_{\cH}.
  \end{split}
  \end{align}
This means
  $(\minop^*-\zbar)g\in \ker(\minop^*-z)^\bot=\Ran(\minop-\zbar)$, where the
  last equality holds 
  because $z\in\Phi_-(\minop^*)$ 
  implies $\Ran(\minop-\zbar)$ is closed, thus 
  $(\minop^*-\zbar)g=(\minop-\zbar)h$ for some  $h\in\dom(\minop)=\ker\tr$.
  Since $S^*$ is an extension of $S$, this implies
  $g-h\in\ker(\minop^*-\zbar)$ and so $u=\tr g=\tr (g-h) \in\cM(\zbar)$,
  as required.

When $s$ is real \eqref{eq:sympl_complement} gives $\cM(s)^\omega = \cM(s)$, so $\cM(s)$ is
  Lagrangian.  It is increasing because the crossing form $\form[v]$
  in \eqref{eq:form_Mdot} is positive definite ($f\neq 0$ for nonzero
  $v$ in \Cref{cor:Mdot}).
  

  \eqref{item:inf_limit}\, Because the lower
  bounds of $\minop$ and its Friedrichs extension $H_F$ coincide,
 we have  $(-\infty, \gamma) \subset \rho (H_F) \subset \Phi_-(\minop^*)$ by
  \eqref{eq:Phi_minus_ess}, therefore $\cM(s)$ is continuous on 
  $(-\infty,\gamma)$ by \Cref{thm:analyticM_general}. From \cite[Cor.~5.2.14]{BehHasDeS_boundarytriples} we have that the
  $s \downarrow -\infty$ limit of $\cM(s)$ exists in the strong
  resolvent sense, and \cite[Thm.~5.5.1]{BehHasDeS_boundarytriples} gives
  $\cM(-\infty) = \cF$.  
  
  For the limit $s \uparrow \gamma$, we first use 
   \cite[Cor.~5.5.5]{BehHasDeS_boundarytriples} to find a boundary triplet 
   $(\cK', \Gamma_0', \Gamma_1')$ such that $\dom(H_F) = \ker \Gamma_0'$.
   It then follows from \cite[Cor.~5.2.14]{BehHasDeS_boundarytriples} that the 
   corresponding Cauchy data space $\cM'(z)$ has a left-hand limit 
   $\cM'(\gamma-)$ in the strong resolvent sense, and this limit is Lagrangian. 
  By
  \cite[Thm.~2.5.1]{BehHasDeS_boundarytriples}, the triplets $(\cK, \Gamma_0, \Gamma_1)$ 
  and $(\cK', \Gamma_0', \Gamma_1')$ are
  related by a bounded symplectic transformation, thus the Cauchy data
  spaces $\cM(z)$ and $\cM'(z)$ are related by a M\"obius
  transform \cite[Eq.~(2.5.4)]{BehHasDeS_boundarytriples}, which
  preserves convergence and the Lagrangian property.
  To complete the proof we note that for Lagrangian subspaces, strong resolvent
  convergence is equivalent to strong graph convergence, by
  \cite[Cor.~1.9.6]{BehHasDeS_boundarytriples}.
%
\end{proof}

\section{First proof of main theorems}
\label{sec:proof1}

We are now ready to prove our main results, namely \Cref{thm:main,thm:shift_Hormander}, 
using the Maslov index. 
We are thus back to the assumption that $\cK$ is finite dimensional. There are three key ingredients in the proof, two of which have already been established:

\begin{enumerate}
	\item A formula for the difference of counting functions in terms of the difference of Maslov indices (\Cref{lem:Maslov_counting});
	\item The identity of Zhou--Wu--Zhu relating the difference of Maslov indices to the 
	difference of Duistermaat indices (formula \eqref{eq:Hor_ZWZ0});
	\item A one-sided continuity result for the Duistermaat index (\Cref{prop:continuous}).
\end{enumerate}

Working towards the counting formula in \Cref{lem:Maslov_counting}, we first relate the eigenvalues 
of a self-adjoint extension to the 
intersections of the corresponding Lagrangian plane with the Cauchy data space
 $\cM(\cdot)$ defined in \eqref{eq:CDS_def}.

\begin{lemma}
  \label{lem:intersection_dim} 
  Under the assumptions of \Cref{thm:main}, let $H_1$ and $H_2$ be self-adjoint
  extensions of $\minop$ corresponding to Lagrangian planes $\cL_1$ and $\cL_2$. If $z \in \bbC \backslash \spess(\minop)$, then
    \begin{equation}
      \label{eq:intersection_ineq}
      \dim \ker(H_j - z) = \dim\!\big(\cM(z) \cap \cL_j\big) + \dim \ker (S-z)
    \end{equation}
    for $j=1,2$, therefore
    \begin{equation}
      \label{eq:intersection_dim}
      \dim \ker(H_1 - z) - \dim \ker(H_2 - z)
      = \dim\!\big(\cM(z) \cap \cL_1\big) - \dim\!\big(\cM(z) \cap \cL_2\big).
    \end{equation}
\end{lemma}

\begin{remark}
  \label{rem:intersectionUCP}
  If $\minop$ has the unique continuation property, so that
  \eqref{c12} holds for all $z \in \C$, then
  \eqref{eq:intersection_ineq} implies
  $\dim \ker(H_j - z) = \dim(\cM(z) \cap \cL_j)$ for all
  $z \in \bbC \backslash \spess(\minop)$. This is no longer true if
  $\minop$ does not have the unique continuation property (see
  \Cref{sec:noUCP} for an elementary example), but
  \eqref{eq:intersection_dim} holds regardless.
  \hfill$\Diamond$
\end{remark}

\begin{proof}
  From the definition of the extension $H_j$ using the Lagrangian
  plane $\cL_j$ in \eqref{eq:domain_extension}, and the definition of
  $\cM(z)$ in \eqref{eq:CDS_recall}, we have
  \begin{equation}
    \label{eq:TrE}
    \tr \big(\! \ker(H_j - z) \big) = \cM(z) \cap \cL_j.
  \end{equation}
  By the rank-nullity theorem,
  \begin{equation}
    \label{eq:mismatch_dimension}
    \dim \ker(H_j - z)
    = \dim\!\big(\cM(z) \cap \cL_j\big)
    + \dim \ker\left(\tr \big|_{\ker(H_j - z)}\right).
  \end{equation}
  In turn,
  \begin{equation}
    \ker\left(\tr \big|_{\ker(H_j - z)}\right)
    = \ker(H_j - z) \cap \ker(\tr)
    = \ker(H_j - z) \cap \dom(S)
    = \ker(S-z).
  \end{equation}
  Substituting this into \eqref{eq:mismatch_dimension} gives
  \eqref{eq:intersection_ineq}.
\end{proof}

We now relate the eigenvalue counting functions and Maslov indices, recalling 
the definition of $N(H;I)$ 
from \eqref{eq:counting_def}.

\begin{proposition}
  \label{lem:Maslov_counting}
With the hypotheses and notation of \Cref{thm:main},
 for any interval $[a,b] \subset \bbR \backslash \spess(\minop)$ we have
  \begin{equation}
    \label{eq:Maslov_counting}
    N\big(H_1; (a,b] \big) - N\big(H_2; (a,b] \big)
    = \Mas_{[a, b]}\!\big(\cL_2, \cM(\cdot)\big)
    - \Mas_{[a, b]}\!\big(\cL_1, \cM(\cdot) \big).
  \end{equation}
\end{proposition}

As mentioned in the introduction, it is only possible to express the counting functions $N\big(H_j; (a,b] \big)$ in terms of Maslov indices if one assumes the unique continuation property of $\minop$, but formula \eqref{eq:Maslov_counting} for their \emph{difference} does not require this assumption; see \Cref{rem:intersectionUCP}.

\begin{proof}
It follows from \Cref{prop:cM_real_s}\eqref{item:complement} and \Cref{cor:Mdot} that $\cM(s)$ is an increasing 
path of Lagrangian planes, so we can use \eqref{eq:Maslov_intersections2}, \eqref{eq:intersection_dim} and 
\eqref{eq:counting_def} to obtain
\begin{align*}
    \Mas_{[a,b]}\!\big(\cL_2, \cM(\cdot)\big) - \Mas_{[a,b]}\!\big(\cL_1, \cM(\cdot) \big)
    &= \sum_{t \in (a,b]} \big[ \dim\!\big(\cM(t)\cap\cL_1 \big) - \dim\!\big(\cM(t)\cap\cL_2 \big) \big] \\
    &= \sum_{\lambda \in (a,b]}  \big[ \dim \ker(H_1 - \lambda) - \dim \ker(H_2 - \lambda) \big] \\
	&= N\big(H_1; (a,b] \big) - N\big(H_2; (a,b] \big)
  \end{align*}
as claimed.
\end{proof}

We are now ready to prove our main results.

\begin{proof}[Proof of \Cref{thm:shift_Hormander}]
For $[a,b] \subset \bbR \backslash \spess(\minop)$ we combine
 \eqref{eq:Maslov_counting} and the
  Zhou--Wu--Zhu identity \eqref{eq:Hor_ZWZ0} to obtain
  \begin{equation}
    \label{eq:Hormander_counting}
    N\big(H_1; (a,b] \big) - N\big(H_2; (a,b] \big)
    = \iD\!\big(\cL_1, \cL_2, \cM(b)\big) - \iD\!\big(\cL_1, \cL_2, \cM(a)\big),
  \end{equation}
which is exactly \eqref{eq:count_triple}.

 For $\lambda$ below the essential spectrum we use \eqref{eq:Hormander_counting} 
 with $b = \lambda$ and $a < b$. \Cref{cor:Mdot} and item \eqref{item:inf_limit} of \Cref{prop:cM_real_s} 
say that $\cM(a)$ converges to $\cF$ from above as $a \to -\infty$, so we can use 
\eqref{eq:limit3} from \Cref{prop:continuous} to compute
\begin{equation}
\label{iotalimit}
	\lim_{a \to -\infty} \iD\!\big(\cL_1, \cL_2, \cM(a)\big) = \iD(\cL_1, \cL_2, \cF )
\end{equation}
and hence arrive at \eqref{eq:shift_triple}.
\end{proof}

\begin{proof}[Proof of \Cref{thm:main}]
We start with \eqref{eq:shift_triple} from \Cref{thm:shift_Hormander},
  \begin{equation}
    \label{eq:shift_triple_again}
    \sigma(H_1,H_2; \lambda)
    = \iD\!\big(\cL_1, \cL_2, \cM(\lambda)\big) - \iD(\cL_1, \cL_2, \cF),
  \end{equation}
  and use the bound \eqref{eq:Dui_estimate} to
  estimate $\iD\!\big(\cL_1, \cL_2, \cM(\lambda)\big)$, arriving at
  \begin{equation}
    \label{eq:shift_bounds}
    - \iD(\cL_1,\cL_2,\cF) 
    \leq \sigma(H_1,H_2; \lambda)
    \leq n - \dim \cL_1 \cap \cL_2 - \iD(\cL_1, \cL_2, \cF).
  \end{equation}
  We then use \eqref{eq:iDswap1} to simplify the right-hand side to $\iD(\cL_2,\cL_1,\cF)$.
\end{proof}

\section{The Duistermaat index of a self-adjoint linear relation}
\label{sec:proof_krein_resolvent}

A Lagrangian plane in 
$(\cK\oplus\cK, \omega)$ can be viewed as a self-adjoint linear
relation; see \cite[Sec.~4.2]{BooZhu_memAMS}.  In this context,
the difference of two Lagrangian planes, $\cL$ and $\cM$, is the
Lagrangian plane
\begin{equation}
	\label{eq:difference_def}
	\cL - \cM := \left\{
	\begin{pmatrix}
		u \\ v_{_\cL} - v_{_\cM}   
	\end{pmatrix}
	\in \cK \oplus \cK :
		\begin{pmatrix}
		u \\ v_{_\cL}   
	\end{pmatrix}
	\in \cL,\
	\begin{pmatrix}
		u \\ v_{_\cM}  
	\end{pmatrix}
	\in \cM \right\},
\end{equation}
the \emph{inverse} of $\cL$ is
\begin{equation}
	\label{eq:inverse_def}
	\cL^{-1} := \left\{
	\begin{pmatrix}
		v \\ u
	\end{pmatrix}
	\in \cK\oplus\cK :
	\begin{pmatrix}
		u \\ v
	\end{pmatrix}
	\in \cL \right\},
\end{equation}
and the \emph{kernel} and \emph{multivalued part} of 
$\cL$ are
\begin{equation}
	\label{eq:ker_def}
	\ker \cL := \left\{u \in \cK :
	\begin{pmatrix}
		u \\ 0
	\end{pmatrix}
	\in \cL \right\},
	\qquad
	\mul \cL := \left\{v \in \cK :
	\begin{pmatrix}
		0 \\ v
	\end{pmatrix}
	\in \cL \right\}.
\end{equation}
The dimension of $\ker \cL$ is called the \emph{nullity} and
denoted $n_0(\cL)$.  If $\mul{\cL} = 0$, then $\cL$ is the
graph of a self-adjoint linear operator on $\cK$. In this case we 
define the \emph{index} $n_-(\cL)$
to be
the number of negative eigenvalues of the corresponding operator.

\begin{proposition}
	\label{prop:diff_of_Krein_terms}
	Assume $\cL_1$, $\cL_2$ and $\cL_3$ are Lagrangian planes such that $\cL_3$ is transversal
	to $\cL_1$, $\cL_2$ and $\cV = 0\oplus\cK$.  Then the Lagrangian
	plane
	\begin{equation}
		\label{eq:diff_of_Krein_terms}
		\Delta := (\cL_1 - \cL_3)^{-1} - (\cL_2 - \cL_3)^{-1}
	\end{equation}
	is a graph of an operator on $\cK$, with nullity
	and index
	\begin{equation}
		\label{eq:index_Delta}
		n_0(\Delta) = \dim \cL_1 \cap \cL_2,
		\qquad
		n_-(\Delta) = \iD(\cL_1,\cL_2,\cL_3).
	\end{equation}
\end{proposition}

\newcommand{\PP}[1]{P_{#1}}

\begin{figure}
  \centering
  \includegraphics[scale=0.9]{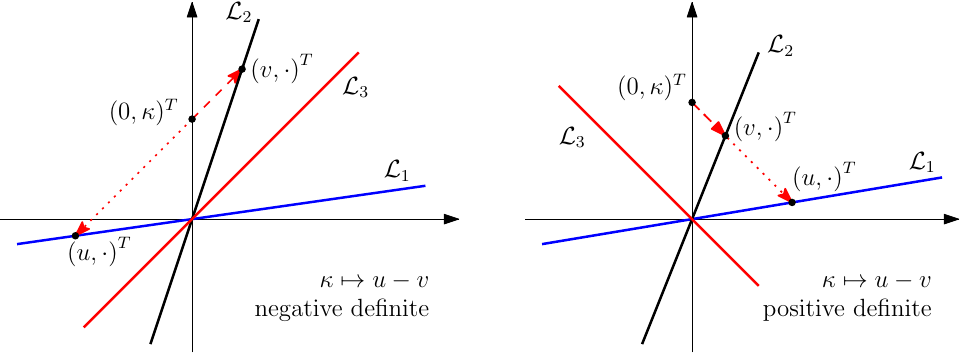}
  \caption{Examples of the action
    $\Delta \colon \kappa \mapsto u(\kappa)-v(\kappa)$ with
    $n_-(\Delta) = 1$ (left) and $n_-(\Delta)=0$ (right).  This should
    be compared with Fig.~\ref{fig:iD_examples}, top row.  Dotted
    and dashed lines illustrate the action of the projectors $\PP1$
    and $\PP2$, respectively.}
  \label{fig:Delta}
\end{figure} 

\begin{proof}
  Let $\PP1$ and $\PP2$ be the
  projections onto $\cL_1$ and $\cL_2$, respectively, parallel to $\cL_3$.
  The transversality conditions on $\cL_3$ imply that their
  restrictions $\PP1\colon \cV \to \cL_1$ and
  $\PP2 \colon \cV \to \cL_2$ are isomorphisms.  Written
  explicitly, they are
  \begin{equation}
    \label{eq:bijV1}
    \PP1 \colon
    \begin{pmatrix}
      0 \\ \kappa
    \end{pmatrix} \in \cV
    \ \mapsto\
    \begin{pmatrix}
      u \\ \alpha
    \end{pmatrix} \in \cL_1
    \quad\text{such that}\quad
    \begin{pmatrix}
      u \\ \alpha
    \end{pmatrix}
    -
    \begin{pmatrix}
      0 \\ \kappa
    \end{pmatrix}
    \in \cL_3,  
  \end{equation}
  and 
  \begin{equation}
    \label{eq:bijV2}
    \PP2 \colon
    \begin{pmatrix}
      0 \\ \kappa
    \end{pmatrix} \in \cV
    \ \mapsto\
    \begin{pmatrix}
      v \\ \beta
    \end{pmatrix} \in \cL_2
    \quad\text{such that}\quad
    \begin{pmatrix}
      v \\ \beta
    \end{pmatrix}
    -
    \begin{pmatrix}
      0 \\ \kappa
    \end{pmatrix}
    \in \cL_3.  
  \end{equation}
  For future use, we observe that
  \begin{equation}
    \label{eq:kerDelta1a}
    \PP1
    \begin{pmatrix}
      0 \\ \kappa
    \end{pmatrix}
    -
    \PP2
    \begin{pmatrix}
      0 \\ \kappa
    \end{pmatrix}
    \in \cL_3,
  \end{equation}
  and
  \begin{equation}
    \label{eq:PPequality}
    \PP1
    \begin{pmatrix}
      0 \\ \kappa
    \end{pmatrix}
    =
    \PP2
    \begin{pmatrix}
      0 \\ \kappa
    \end{pmatrix}
    \qquad\Longleftrightarrow\qquad
    \PP1
    \begin{pmatrix}
      0 \\ \kappa
    \end{pmatrix}
    \in \cL_1 \cap \cL_2.
  \end{equation}
  The non-trivial $\Leftarrow$ direction in
  \eqref{eq:PPequality} follows from the observation that if
  $(u,\alpha)^T$ in \eqref{eq:bijV1} belongs to both $\cL_1$ and
  $\cL_2$, then it satisfies the condition in \eqref{eq:bijV2} and
  therefore coincides with the unique value of $(v,\beta)^T$.
  
  We now compute the Lagrangian plane $\Delta$ from~\eqref{eq:diff_of_Krein_terms}.  Using definition
  \eqref{eq:difference_def}, we write
  \begin{align}
    \nonumber
    \cL_1-\cL_3
    &= \left\{
      \begin{pmatrix}
        u \\ \alpha-\gamma
      \end{pmatrix}
    \in \cK\oplus\cK
    \colon
    \begin{pmatrix}
      u \\ \alpha
    \end{pmatrix}
    \in \cL_1,\
    \begin{pmatrix}
      u \\ \gamma
    \end{pmatrix}
    \in\cL_3\right\}
    \\[2pt]
    &=
      \left\{
      \begin{pmatrix}
        u \\ \kappa
      \end{pmatrix}
    \colon
    \exists \alpha\in\cK \,\text{ such that }
    \begin{pmatrix}
      u \\ \alpha  
    \end{pmatrix}
    \in\cL_1,\,
    \begin{pmatrix}
      u \\ \alpha-\kappa  
    \end{pmatrix}
    \in\cL_3
    \right\},
    \label{eq:L1L3}
  \end{align}
  where the vector $(u,\alpha)^T$ is exactly the image of
  $(0,\kappa)^T$ under $\PP1$.  Defining $\pi \colon \cK\oplus\cK \to \cK$
  to be the projection onto the first component and using
  \eqref{eq:inverse_def}, we arrive at
  \begin{equation}
    \label{eq:L1L3_inv}
    (\cL_1-\cL_3)^{-1} =
    \left\{
      \begin{pmatrix}
        \kappa \\  u(\kappa)
      \end{pmatrix}\in\cK\oplus\cK
      \colon
      \ \kappa\in\cK,
    \right\},
    \qquad \text{where }
    u(\kappa) := \pi \PP1
    \begin{pmatrix}
      0 \\ \kappa
    \end{pmatrix}.
  \end{equation}  
  Analogous considerations for $\cL_2$ result in
  \begin{equation}
    \label{eq:Delta_uv}
    \Delta =
    \left\{
      \begin{pmatrix}
        \kappa \\  u(\kappa) - v(\kappa)
      \end{pmatrix}\in\cK\oplus\cK
      \colon \kappa\in\cK \right\},
    \qquad
    \text{where }
    v(\kappa) := \pi \PP2
    \begin{pmatrix}
      0 \\ \kappa
    \end{pmatrix}.
  \end{equation}

  In other words, $\Delta$ is the graph of an operator mapping
  $\kappa \in\cK$ to $u(\kappa)-v(\kappa)\in\cK$, see
  Fig.~\ref{fig:Delta} for an example.  The kernel of $\Delta$ is the
  space of $\kappa$ such that the corresponding $u(\kappa)$ and
  $v(\kappa)$ coincide.  More precisely,
  \begin{align}
    \begin{pmatrix}
      \kappa \\ 0
    \end{pmatrix}
    \in \ker\Delta
    \quad
    &\Longleftrightarrow\quad
      \pi \PP1
      \begin{pmatrix}
        0 \\ \kappa
      \end{pmatrix}
    -
    \pi \PP2
    \begin{pmatrix}
      0 \\ \kappa
    \end{pmatrix}
    =0
    \nonumber \\
    \label{eq:kerDelta1}
    &\Longleftrightarrow \quad
      \PP1
      \begin{pmatrix}
        0 \\ \kappa
      \end{pmatrix}
    -
    \PP2
    \begin{pmatrix}
      0 \\ \kappa
    \end{pmatrix}
    \in \cV.
  \end{align}
  Combining this with \eqref{eq:kerDelta1a} and the condition
  $\cV\cap\cL_3 = 0$, we obtain
  \begin{equation}
    \label{eq:kerDelta1b}
    \begin{pmatrix}
      \kappa \\ 0
    \end{pmatrix}
    \in \ker\Delta
    \quad\Longleftrightarrow\quad
    \PP1
    \begin{pmatrix}
      0 \\ \kappa
    \end{pmatrix}
    =
    \PP2
    \begin{pmatrix}
      0 \\ \kappa
    \end{pmatrix}.
  \end{equation}
  Finally, \eqref{eq:PPequality} yields that $\ker\Delta$ is
  isomorphic to $\cL_1 \cap \cL_2$.
    
  We now calculate the Duistermaat index using \Cref{cor:n_minus_Q},
  namely by evaluating the Morse index of $Q(\cL_1,\cL_3;\cL_2)$.
  The mapping $L \colon \cL_1\to\cL_3$ appearing in the definition
  \eqref{eq:Qdef} of $Q(\cL_1,\cL_3;\cL_2)$ acts as
  \begin{equation}
    \label{eq:Lmap_13}
    \cL_1 \ni
    \begin{pmatrix}
      u \\ \alpha
    \end{pmatrix}
    =
    \PP1
    \begin{pmatrix}
      0 \\ \kappa
    \end{pmatrix}
    \quad
    \mapsto
    \quad
    \PP2
    \begin{pmatrix}
      0 \\ \kappa
    \end{pmatrix}
    -
    \PP1
    \begin{pmatrix}
      0 \\ \kappa
    \end{pmatrix} \in \cL_3,
  \end{equation}
  cf.\ \eqref{eq:Qdef}, \eqref{eq:kerDelta1a} and the fact that $P_2$
  acts into $\cL_2$.
 Consequently, for any two vectors from $\cL_1$,
  \begin{equation}
    \label{eq:ujkj}
    \begin{pmatrix}
      u_1 \\ \alpha_1
    \end{pmatrix}
    =
    \PP1
    \begin{pmatrix}
      0 \\ \kappa_1
    \end{pmatrix},
    \qquad
    \begin{pmatrix}
      u_2 \\ \alpha_2
    \end{pmatrix}
    =
    \PP1
    \begin{pmatrix}
      0 \\ \kappa_2
    \end{pmatrix},
  \end{equation}
 the form $Q$ acts as
  \begin{equation}
    \label{eq:ZWZ_Q_def}
    \left(
      \begin{pmatrix}
        u_1 \\ \alpha_1
      \end{pmatrix},
      \begin{pmatrix}
        u_2 \\ \alpha_2
      \end{pmatrix}
    \right)
    \ \mapsto\
    \omega \left(
      \PP1
      \begin{pmatrix}
        0 \\ \kappa_1
      \end{pmatrix},
      \PP2
      \begin{pmatrix}
        0 \\ \kappa_2
      \end{pmatrix}
      -
      \PP1
      \begin{pmatrix}
        0 \\ \kappa_2
      \end{pmatrix}
    \right).
  \end{equation}

  We now note that for the purpose of computing the index, the form
  $Q$ can be viewed as a Hermitian form on $\cV$. Observing that
  \begin{equation}
    \label{eq:Qexpand1}
    \omega\left(
      \PP1
      \begin{pmatrix}
        0 \\ \kappa_1
      \end{pmatrix}
      -
      \begin{pmatrix}
        0 \\ \kappa_1
      \end{pmatrix},\ 
      \PP2
      \begin{pmatrix}
        0 \\ \kappa_2
      \end{pmatrix}
      - \PP1
      \begin{pmatrix}
        0 \\ \kappa_2
      \end{pmatrix}
    \right) = 0,
  \end{equation}
  which holds because both arguments are in the Lagrangian plane
  $\cL_3$, we have for the value of $Q$,
  \begin{align}
    \label{eq:Qalt}
    \omega \left(
      \PP1
      \begin{pmatrix}
        0 \\ \kappa_1
      \end{pmatrix},
      \PP2
      \begin{pmatrix}
        0 \\ \kappa_2
      \end{pmatrix}
      -
      \PP1
      \begin{pmatrix}
        0 \\ \kappa_2
      \end{pmatrix}
    \right)
    &= \omega \left(
      \begin{pmatrix}
        0 \\ \kappa_1
      \end{pmatrix},\ 
      \PP2
      \begin{pmatrix}
        0 \\ \kappa_2
      \end{pmatrix}
      - \PP1
      \begin{pmatrix}
        0 \\ \kappa_2
      \end{pmatrix}
    \right)
    \\ \nonumber
    &= \left<\kappa_1, u(\kappa_2) - v(\kappa_2) \right>_\cK,
  \end{align}
  which is exactly the sesquilinear form corresponding to
  $\Delta\colon \kappa \mapsto u(\kappa) - v(\kappa)$, see
  \eqref{eq:Delta_uv}.  By \Cref{cor:n_minus_Q}, we conclude that
  \begin{equation}
    \label{eq:iDQh}
    \iD(\cL_1,\cL_2,\cL_3) = n_-\!\left(\Delta\right),
  \end{equation}
  which establishes the desired result.
\end{proof}

Next we provide a proof of \Cref{thm:resolvent_diff} by virtue of a slightly more general result. 

\begin{theorem}
  \label{thm:resolvent_diff_new}
  Assume the setting of Theorem \ref{thm:main}. For
  $\lambda\in \rho(H_1)\cap\rho(H_2) \cap \bbR$, the operator
  $D(\lambda)$, defined in \eqref{eq:resolvent_diff}, is reduced by
  the decomposition $\cH=\cN_{\lambda}\oplus \cN_{\lambda}^{\perp}$,
  where $\cN_{\lambda}:=\ker(\minop^*-\lambda)$.  The block form of
  $D(\lambda)$ with respect to this decomposition is
  \begin{equation}
    \label{6.22}
    D(\lambda)=\begin{pmatrix}
      K(\lambda) 
      &0 \\
      0 &0
    \end{pmatrix},
  \end{equation}
  where $K(\lambda) \colon \cN_{\lambda} \to \cN_{\lambda}$ is the
  restriction of $D(\lambda)$ to $\cN_{\lambda}$. For an arbitrary
  interval $I\subset \rho(H_1) \cap \rho(H_2) \cap \mathbb R$, the
  eigenvalues of $K(\lambda)$ depend continuously on $\lambda\in I$
  and one has
  \begin{align}
    \label{eq:index_K0}
    n_0\big(K(\lambda)\big) &= \dim\cL_1\cap \cL_2,
    \\
    \label{eq:index_K-}
    n_-\big(K(\lambda)\big) &= n_-\big(D(\lambda)\big)
                              = \iD\!\big(\cL_1, \cL_2, \cM(\lambda)\big),
    \\
    \label{eq:index_K+}
    n_+\big(K(\lambda)\big) &= n_+\big(D(\lambda)\big)
                              = \iD\!\big(\cL_2, \cL_1, \cM(\lambda)\big).
  \end{align}
  In particular, $K(\lambda)$ has constant rank
  \begin{equation}\label{1.27}
    r=n-\dim \cL_1\cap \cL_2=\iD\!\big(\cL_1, \cL_2, \cM(\lambda)\big)
    +\iD\!\big(\cL_2, \cL_1, \cM(\lambda)\big)
  \end{equation} 
  and the functions $\lambda\mapsto n_{\pm}(K (\lambda))$ are constant
  on $I$.
\end{theorem}

\begin{proof}
  Since $H_j \subset \minop^*$, we have $(\minop^*-\lambda)
  (H_j-\lambda)^{-1} =I_\cH$ for $j=1,2$, therefore
  $(\minop^*-\lambda)D(\lambda)=0$, i.e., $\Ran D(\lambda) \subset
  \cN_{\lambda}$.  This establishes the second row of \eqref{6.22}
  and, by self-adjointness, the whole of \eqref{6.22}.

To prove the continuity of the eigenvalues of $K(\lambda)$ (acting on the $\lambda$-dependent spaces $\cN_\lambda$) we first note that $\minop^*-\lambda$ is onto in $\cH$ as an extension of a surjective operator $H_1-\lambda$ with $\lambda\in\rho(H_1)$. Therefore, by steps three and four in the proof of \Cref{thm:analyticM_general}, there exists a continuous family of bijections $F_{\lambda}$ mapping $\cN_{\lambda_0}$ onto $\cN_{\lambda}$. The family of operators $\lambda\mapsto F_{\lambda}^{-1}K(\lambda)F_{\lambda}$ acting on the $\lambda$-independent Hilbert space $\cN_{\lambda_0}$ is continuous and for each $\lambda\in I$ the eigenvalues of $K(\lambda)$ and $F_{\lambda}K(\lambda)F_{\lambda}^{-1}$ coincide. These facts yield the desired continuity assertion.

We next show that \eqref{eq:index_K0} holds for all $\lambda \in \rho(H_1) \cap \rho(H_2) \cap \mathbb R$. For this we use the resolvent difference formula  from \cite[Thm.~2.5]{LatSuk_prep20},
	\begin{equation}\lb{newKrein}
		K(\lambda)= \big(\tr (H_1-\lambda)^{-1} \big)^* P_1 J P_2 \tr (H_2-\lambda)^{-1}\big|_{\cN_{\lambda}},
	\end{equation}
	where 
	$P_j$ denotes the orthogonal projection onto $\cL_j$ in $\cK\oplus\cK$. 
	By \Cref{propositionc.16} the maps 
	\begin{equation}
	\tr (H_2-\lambda)^{-1}\colon \cN_{\lambda}\rightarrow \cL_2,\qquad \big(\tr (H_1-\lambda)^{-1}\big)^*\colon \cL_1\rightarrow\cN_{\lambda},
	\end{equation}
are bijective. 
 Using \eqref{newKrein} and the identity $P_1 J = J(I-P_1)$ from \eqref{PJ}, 
we get
	\begin{equation}
		\dim\ker K(\lambda)=\dim\ker \big(P_1 J |_{\cL_2} \big) 
		= \ker \big((I-P_1) |_{\cL_2} \big)=\cL_1\cap \cL_2,
	\end{equation}
	proving \eqref{eq:index_K0}.
	
It follows that $n_0\big(K(\lambda)\big)$ is constant on the interval $I$, so the same is true of 
$n_\pm\big(K(\lambda)\big)$, since the eigenvalues are continuous. It thus suffices to prove 
\eqref{eq:index_K-} and \eqref{eq:index_K+} for a single $\lambda \in I$. We will choose 
$\lambda\in I \backslash \spec(H_0)$, where $H_0$ is the extension of $\minop$ with 
 $\tr(\dom(H_0))=\ker\Gamma_0$. Such a $\lambda$ always exists
 because $I\subset\rho(H_1)\cap\rho(H_2)$ does not intersect $\spess(H_j) = \spess(\minop)$, 
and therefore $I$ can only contain isolated eigenvalues of $H_0$.

For $\lambda\in I\backslash \spec(H_0)$ we may apply the classical Krein--Naimark formula,  
\begin{equation}
	(H_j - \lambda)^{-1} - (H_0-\lambda)^{-1} = \gamma(\lambda) \big(\cL_j - \cM(\lambda) \big)^{-1}\gamma({\lambda})^*,
\end{equation}
where $\gamma(\lambda):=(\Gamma_0|_{\cN_{\lambda}})^{-1}\in \mathcal B (\cK, \cH)$ is the $\gamma$-field; see, for example, \cite[Thm.~14.18]{Schmudgen_unboundedSAO}.
Applying this to $D(\lambda)=(H_1-\lambda)^{-1}-(H_2-\lambda)^{-1}$ yields
\begin{equation}\lb{5.2new}
	D(\lambda) = \gamma(\lambda)\left((\cL_1-\cM(\lambda))^{-1}-(\cL_2-\cM(\lambda))^{-1}\right)\gamma({\lambda})^*.
\end{equation}
Employing \Cref{prop:diff_of_Krein_terms} with $\cL_3=\cM(\lambda)$ (the transversality conditions are satisfied since $\lambda$ is not an eigenvalue of $H_1$, $H_2$ or $H_0$) together with the fact that $\gamma(\lambda) \colon \cK\rightarrow \cN_{\lambda}$ is a bijection (see \cite[Lem.~14.13(ii)]{Schmudgen_unboundedSAO}), we get
\begin{align}
	\begin{split}\label{eq:i12new}
		\iD\!\big(\cL_1, \cL_2, \cM(\lambda)\big)&=n_-\!\left(\big(\cL_1-\cM(\lambda)\big)^{-1} - \big(\cL_2-\cM(\lambda) \big)^{-1}\right) 
		=n_-(D(\lambda)), 
	\end{split}
\end{align}
which is exactly \eqref{eq:index_K-}.
To complete the proof we note that \eqref{1.27} follows from \eqref{eq:index_K0} and the identity \eqref{eq:iDswap1} for the Duistermaat index. Combining this with \eqref{eq:index_K-} yields \eqref{eq:index_K+}.
\end{proof}

\begin{lemma}\label{propositionc.16}Assume that $H$ is a self-adjoint
  extension of $\minop$ corresponding to the Lagrangian plane
  $\cL\subset\cK\oplus \cK$. For $\lambda \in \rho(H)$, we let
  $R_\lambda := (H-\lambda)^{-1}$ and consider $\tr R_{\lambda}\colon
  \cH\rightarrow \cK\oplus \cK$.  Then
	\begin{equation}\lb{c.21}
	\ker(\tr R_{\lambda})=\overline{\ran(\minop-\lambda)},\qquad \ran(\tr R_{\lambda})=\cL.
	\end{equation}
In particular, $\tr R_{\lambda}$ is a bijection between $\ker(\minop^*-\overline{\lambda})$ and $\cL$.
\end{lemma}

\begin{proof}
	To show the inclusion $\ker(\tr R_{\lambda})\subset{\ran(\minop-\lambda)}$, suppose that $\tr R_{\lambda}u=0$. Then by \eqref{c.8} one has $R_{\lambda}u=v$ for some $v\in\dom(S)$, hence $u=(S-\lambda)v$ as required. To prove $\ker(\tr R_{\lambda})\supset{\ran(\minop-\lambda)}$, we note that $R_{\lambda}(\minop-\lambda)v=v$ and use \eqref{c.8}. For the second identity we note that $\Ran(\tr R_{\lambda})=\tr(\dom(H))=\cL$. 
\end{proof}

Finally, we explain how this implies the  interlacing formulas in \Cref{thm:main}.

\begin{proof}[Proof of \Cref{thm:main} using \Cref{thm:resolvent_diff}]
Using \eqref{iotalimit}, we can choose a large negative number $\lambda_*$ that satisfies 
\begin{equation}
  \label{eq:indices_well_below}
  \iD\!\big(\cL_1, \cL_2, \cM(\lambda_*)\big)
  = \iD(\cL_1, \cL_2, \cF ) = \smin,
  \qquad 
  \iD\!\big(\cL_2, \cL_1, \cM(\lambda_*)\big)
  = \iD(\cL_2, \cL_1, \cF ) = \smax
\end{equation}
and is below the spectra of $H_1$ and $H_2$. For such $\lambda_*$ the
eigenvalues of $R_1 := (\lambda_*-H_1)^{-1}$ are negative and bounded;
we label them in increasing order as $\mu_k(R_1)$, and likewise
for $R_2 := (\lambda_*-H_2)^{-1}$.  Recalling the definition of
$D(\lambda)$ in \eqref{eq:resolvent_diff}, we write
$R_2 = R_1 + D(\lambda_*)$ and note that $n_-\big(D(\lambda_*)\big) =
\smin$ and $n_+\big(D(\lambda_*)\big) = \smax$
by \Cref{thm:resolvent_diff}. Applying Weyl interlacing for additive
finite-rank perturbations, we get
\begin{equation}
  \label{muinterlacing}
  \mu_{k - \smin}(R_1) \leq \mu_k(R_2) \leq \mu_{k + \smax}(R_1).
\end{equation}
The corresponding eigenvalues of $H_{j}$ are computed by the
monotone increasing trasformation
$\lambda_k(H_j) = \lambda_* - 1/\mu_k(R_j)$, yielding
\eqref{eq:mainWeylInterlacing}.
\end{proof}

\section{Examples and applications}
\label{sec:examples}

Having proved \Cref{thm:main}, we now discuss some of its consequences. In particular, 
we compare a variety of boundary conditions for a Schr\"odinger operator on the interval 
$(0,1)$ and derive a counting formula of Behrndt and Luger for the Laplacian on quantum 
graphs. We also demonstrate that the upper and lower bounds in \eqref{eq:spec_shift_bound} 
are optimal. Finally, we give an example of an operator not satisfying the UCP, 
where we see that the Maslov index undercounts the eigenvalues even though our theorem 
remains valid.

\subsection{Examples with $n=2$}
\label{sec:examples_n2}

\begin{table}[t]
  \centering
  \footnotesize
  \bgroup
  \def\arraystretch{1.2}
  \begin{tabular}{|c|c|c|c|}
    \hline
    Name
    & Conditions
    & Frame
    & Notation \\ \hline
    Periodic
    & 
    $\begin{array}{l}
        f(0)=f(1),\\  f'(0) = f'(1)
      \end{array}$
    & $X = \begin{pmatrix}
      1 & 0 \\ 1 & 0
    \end{pmatrix}, \ 
                   Y =
                   \begin{pmatrix}
                     0 & 1 \\ 0 & -1
                   \end{pmatrix}$
    & $\cL_{\mathrm{per}}$ \\ \hline
    $\delta$-type
    &
      $\begin{array}{l}
        f(0)=f(1),\\ f'(0)-f'(1) = s f(0)
      \end{array}$
    & $X = \begin{pmatrix}
      1 & 0 \\
      1 & 0
    \end{pmatrix},\ 
          Y =
          \begin{pmatrix}
            s & 1 \\ 0 & -1
          \end{pmatrix}$
    & $\cL_\delta(s)$
    \\ \hline
    Antiperiodic
    & $
      \begin{array}{l}
        f(0)=-f(1),\\ f'(0) = -f'(1)
      \end{array}$
    & $X =
      \begin{pmatrix}
        1 & 0 \\
        -1 & 0
      \end{pmatrix},\
             Y =
             \begin{pmatrix}
               0 & 1 \\ 0 & 1
             \end{pmatrix}$
    & $\cL_{\mathrm{aper}}$  \\ \hline
    $\delta'$-type
    & $
      \begin{array}{l}
        f(0)+f(1) = s f'(0),\\  f'(0) = -f'(1)
      \end{array}$
    & $X =
      \begin{pmatrix}
        1 & s \\
        -1 & 0
      \end{pmatrix},\
             Y =
             \begin{pmatrix}
               0 & 1 \\ 0 & 1
             \end{pmatrix}$
    & $\cL_{\delta'}(s)$  \\ \hline
  \end{tabular}
  \egroup
  \caption{Some commonly encountered boundary conditions and the
    corresponding Lagrangian frames. Note that periodic conditions are
    also called \emph{Neumann--Kirchhoff} or \emph{standard}
    conditions in the context of quantum graphs.}
  \label{tab:BCs}
\end{table}

The prototypical example with $n=2$ is an interval $(0,1)$ with the
Schr\"odinger operator $\minop = -\frac{d^2}{dx^2} + q(x)$. 
However, the results below apply equally
well to a compact quantum graph \cite{BerKuc_graphs} with self-adjoint
conditions imposed everywhere except for two vertices of degree one.  More sophisticated contexts 
 include manifolds with conical singularities \cite{GilKraMen_cjm07}  and \v{S}eba
billiards
\cite{Seb_prl90,BogGirSch_pre02,RahFis_n02,KeaMarWin_jmp10,KurUeb_gafa14,KurUeb_imrn23}
with two or more delta potentials.

For $\minop = -\frac{d^2}{dx^2} + q(x)$, with potential $q \in L^\infty(0,1)$, we have 
$\cH = L^2(0,1)$, $\dom(\minop) = H^2_0(0,1)$ and $\dom(\minop^*) = H^2(0,1)$. 
Both defect numbers are 2, so the self-adjoint extensions are parameterized by Lagrangian planes 
in $\bbC^4$, see \cite[Ex.~14.2 and~14.10]{Schmudgen_unboundedSAO}. The traditional choice
of traces is
\begin{equation}
  \label{eq:traces_interval}
  \Gamma_0 f =
  \begin{pmatrix}
    f(0) \\ f(1)
  \end{pmatrix},
  \qquad
  \Gamma_1 f =
  \begin{pmatrix}
    f'(0) \\ -f'(1)
  \end{pmatrix}.
\end{equation}
With this choice, the Friedrichs extension corresponds to the
vertical plane, $\cF = \cV$. 

We now use  \Cref{thm:main} to compare 
the boundary conditions listed in Table~\ref{tab:BCs}.

\subsubsection{Periodic vs $\delta$-type conditions}
At $s=0$ the $\delta$-type condition
reduces to the periodic condition, $\cL_\delta(0) = \cL_{\mathrm{per}}$, so 
it suffices to consider $s \neq 0$. In this case the rank of the perturbation is 
$2-\dim \big(\cL_{\mathrm{per}}\cap\cL_\delta(s)\big) = 1$. 
The
corresponding $\epsilon$-Robin maps are 
\begin{equation}
  \label{eq:Rperiodic}
  R_\mathrm{per}^\epsilon =
  \frac{1}{2\epsilon}
  \begin{pmatrix}
    1 & -1 \\ -1 & 1
  \end{pmatrix},
  \qquad
  R^\epsilon_\delta = \frac{1}{\epsilon(2 + \epsilon s)}
  \begin{pmatrix}
    1 +  \epsilon s  & -1 \\
    -1 & 1 + \epsilon s
  \end{pmatrix},
\end{equation}
therefore 
\begin{equation}
  \label{eq:R21}
  R^\epsilon_\delta - R^\epsilon_\mathrm{per}
  = \frac{s}{2(2 + \epsilon s)}
  \begin{pmatrix}
    1 & 1 \\ 1 & 1
  \end{pmatrix}.
\end{equation}
Using \Cref{prop:Dui_index}, we conclude that
\begin{equation}
  \iD\!\big(\cL_\mathrm{per},\cL_\delta(s),\cF\big) =
  \begin{cases}
    0, & s > 0, \\ 1, & s < 0,
  \end{cases}
  \qquad\text{and}\qquad
  \big(\smin, \smax\big)
  =
  \begin{cases}
    (0,1), & s > 0, \\
    (1,0), & s < 0,
  \end{cases}
\end{equation}
recovering a well-known result
\cite{BerKuc_incol12,BerKenKurMug_tams19}.

\subsubsection{Periodic vs antiperiodic}

This important case arises in the spectral analysis of Hill's operator and
certain other $\Z$-periodic quantum graphs
\cite{ExnKucWin_jpa10} (namely, those with one edge crossing the boundary of
  the fundamental domain).

For the antiperiodic conditions, we have
\begin{equation}
  \label{eq:R3}
  R^\epsilon_\mathrm{aper} =
  \frac{1}{2\epsilon}
  \begin{pmatrix}
    1 & 1 \\ 1 & 1
  \end{pmatrix},
  \qquad 
  R^\epsilon_\mathrm{aper} - R^\epsilon_\mathrm{per}
  = \frac{1}{\epsilon}
  \begin{pmatrix}
    0 & 1 \\ 1 & 0
  \end{pmatrix}.
\end{equation}
We therefore get
\begin{equation}
  \label{eq:per_aper}
  \big(\smin, \smax\big)
  = (1,1),
\end{equation}
which agrees, for instance, with the interlacing in Hill's equation
\cite[Eq.~(2.4) in Thm.~2.1]{MagnusWinkl_hills}.

\subsubsection{Antiperiodic vs $\delta'$-type conditions}

The antiperiodic conditions are a special case of the $\delta'$-type
conditions with $s=0$.  For $s\neq0$ we have
$\dim\big(\cL_\mathrm{aper} \cap \cL_{\delta'}(s)\big) = 1$, so the rank of the
perturbation is $1$.  Computing the $\epsilon$-Robin map, we obtain
\begin{equation}
  \label{eq:R4}
  R^\epsilon_{\delta'}
  = \frac{1}{s + 2\epsilon}
  \begin{pmatrix}
    1 & 1 \\ 1 & 1
  \end{pmatrix},
  \qquad
  R^\epsilon_\mathrm{aper} -   R^\epsilon_{\delta'}
  = -\frac{s}{2\epsilon(s + 2\epsilon)}
  \begin{pmatrix}
    1 & 1 \\ 1 & 1
  \end{pmatrix}.
\end{equation}
For $s \neq 0$ we see that
$-\frac{s}{2\epsilon(s + 2\epsilon)} \approx -\frac{1}{2\epsilon}$ is
negative for $0 < \epsilon \ll1 $, and hence $\smin = 1$.
Taking the rank of the perturbation into account, we get
\begin{equation}
  \label{eq:aper_deltaprime}
  \big(\smin, \smax\big)
  = (1,0),
\end{equation}
or, in terms of the eigenvalues,
\begin{equation}
  \label{eq:aper_dprime_eig}
  \lambda_{k-1}(H_\mathrm{aper}) \leq \lambda_k\big(H_{\delta'}(s)\big)
  \leq \lambda_k(H_\mathrm{aper}).
\end{equation}

This may be somewhat unexpected, since it implies that 
$\lambda_k\big(H_{\delta'}(s)\big)$ achieves its maximum at $s=0$, while the 
eigenvalue variation formulas of \cite{LatSuk_prep20} can be used to
show that
\begin{equation}
  \label{eq:HFLS}
  \frac{d\lambda}{ds} = - |f_s'(0)|^2,
\end{equation}
where $f_s$ is the normalized eigenfunction corresponding to the
eigenvalue $\lambda(s)$ (assuming it is simple). These two facts are 
reconciled by the observation that $\lambda_1(s) \to
-\infty$ as $s\to0-$, therefore
the \emph{ordered} eigenvalue curves $s \mapsto \lambda_k\big(H_{\delta'}(s)\big)$ are 
discontinuous at $s=0$, as shown in Fig.~\ref{fig:eig_deltaprime}.

\begin{figure}[t]
  \centering
  \includegraphics[scale=0.75]{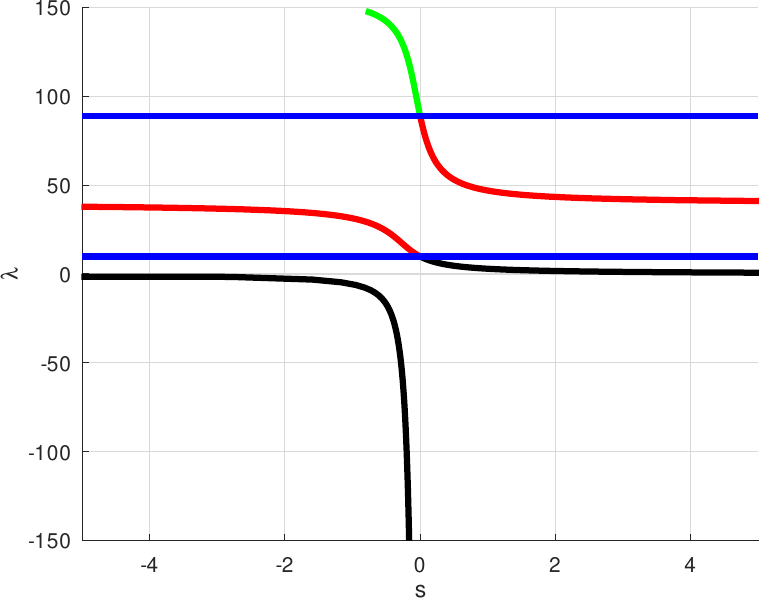}
  \caption{The first five eigenvalues of $H_{\delta'}(s)$, plotted as functions 
  of $s$.  The
    curves are colored according to the index of the eigenvalue:
    $\lambda_1$ is black, $\lambda_2$ is blue, $\lambda_3$ is red
    etc. In particular, the blue lines (from bottom to top) are 
    $\lambda_2\big(H_{\delta'}(s)\big) = \lambda_{1,2}(H_{\rm aper})$
    and $\lambda_4\big(H_{\delta'}(s)\big) = \lambda_{3,4}(H_{\rm aper})$.}
  \label{fig:eig_deltaprime}
\end{figure}

\subsection{Counting negative eigenvalues: the Behrndt--Luger formula}
\label{sec:BL}

In \cite{BehLug_jpa10}, Behrndt and Luger derived a convenient formula
for the number of negative eigenvalues of the Laplacian on a metric
graph.  Here we show how their formula can be obtained from our
results.

We first recall that any Lagrangian plane $\cL$ can be described by a
frame of the form $(P,\, P \Theta P + P - I )$, where
$P\colon \cK\to\cK$ is an orthogonal projector and
$\Theta\colon \Ran P \to \Ran P$ is a self-adjoint operator; see
\Cref{sec:param}. This corresponds to imposing the Dirichlet-type
condition $(I-P) \Gamma_0 f = 0$ and the Robin-type condition
$P \Gamma_1 f = \Theta P \Gamma_0 f$.

\begin{corollary}
  \label{thm:BL_formula}
  Within the setting of \Cref{thm:main}, assume $S$ is non-negative\footnote{Equivalently, 
  the Friedrichs extension $H_F$ is non-negative.} and
  let $H$ be an extension of $S$ specified by a Lagrangian plane
  $\cL$.  Then the Morse index of $H$ is
  \begin{equation}
    \label{eq:Morse_Duistermaat}
    n_-(H) = \iD\!\big( \cM(0-), \cL, \cF \big),
  \end{equation}
  where $\cM(0-)$ is the left-hand limit of $\cM(s)$, as in~\Cref{prop:cM_real_s}.

Furthermore, if $\dom(H_F) = \ker \Gamma_0$ and $\cM(0-)$ is a graph of an 
operator $M_0\colon \cK \to \cK$, then
  \begin{equation}
    \label{eq:BL_formula}
    n_-(H) = n_+(PM_0P - \Theta),
  \end{equation}
  where $(P,\, P \Theta P + P - I )$ is a Lagrangian frame for $\cL$.
\end{corollary}

\begin{remark}
  \label{rem:assumptions_BL_formula}
  Equation~\eqref{eq:Morse_Duistermaat} generalizes results of
  Derkach--Malamud \cite{DerMal_jfa91} in the setting of finite defect
  numbers.  In particular, \cite[Eq.~(0.4)]{DerMal_jfa91} equates the
  Morse indices of $H$ and the operator $\Delta$ that appears in
  \Cref{prop:diff_of_Krein_terms}; thus under more
  restrictive assumptions one can see that
  \cite[Eq.~(0.4)]{DerMal_jfa91} and \eqref{eq:Morse_Duistermaat} give
  the same expressions for $n_-(H)$.
  However, the basis-independent nature of the Duistermaat index
  (manifested as symplectic invariance
  \eqref{eq:symplectic_invariant}) saves one from superfluous
  restrictions such as the form domain inclusion condition in
  \cite[Thms.~5 and 6]{DerMal_jfa91}. \hfill $\diamond$
\end{remark}

 \begin{remark}
 In the setting of \cite{BehLug_jpa10}, where $\minop^*$ is the Laplacian on a metric graph
  and $\Gamma_0$ and $\Gamma_1$ are the standard Dirichlet and
  Neumann traces, the limit $\cM(0-)$ is the graph of an operator 
  $M_0$, by  \cite[Lem.~3]{BehLug_jpa10}, thus \eqref{eq:BL_formula} holds.  To relate
  this to \cite[Thm.~1]{BehLug_jpa10} we simply take
  $\Theta = -L$ and observe that the boundary condition in \cite{BehLug_jpa10} is written in terms of a
  co-frame rather than a frame (see \Cref{sec:param}). In this setting $M_0$ is
  the Dirichlet-to-Neumann map at
  $\lambda=0$, which is explicitly computable; see also \cite[Sec
  3.5]{BerKuc_graphs}.  \hfill $\diamond$
\end{remark}

%

\begin{proof}
  Since $H_F$ is non-negative, we have $\spess(H_F) \subset [0,\infty)$.
  Therefore, we can apply \Cref{thm:shift_Hormander} with 
  $\cL_2=\cF$ and negative $\lambda$ to obtain
    \begin{equation}
    \label{eq:BL_calculation}
	N\big(H_1; (-\infty,\lambda] \big)
    = \iD\!\big(\cL_1, \cF, \cM(\lambda)\big) - \iD(\cL_1, \cF, \cF) 
    =  \iD\!\big(\cL_1, \cF, \cM(\lambda)\big),
  \end{equation}
  using \eqref{eq:iD_special_cases} to eliminate $\iD(\cL, \cF, \cF)$.
  By \Cref{prop:cM_real_s}, $\cM(\lambda)$ converges to a Lagrangian
  plane $\cM(0-)$ as $\lambda \to 0-$.  Since the convergence is
  monotone, \eqref{eq:limit3alt} gives
  $\iD\!\big(\cL_1, \cF, \cM(\lambda)\big) = \iD\!\big(\cM(0-), \cL_1,
  \cF\big)$ for sufficiently small $\lambda$. It follows from \eqref{eq:BL_calculation} that 
  $N\big(H_1; (-\infty,\lambda] \big)$ is constant for small negative $\lambda$ 
  and hence is equal to $n_-(H_1)$, proving equation~\eqref{eq:Morse_Duistermaat}.

  Combining \eqref{eq:Morse_Duistermaat} with
  \Cref{prop:BL_index_formula}, we obtain
  \begin{equation}
    n_-(H) = \iD\!\big( \cM(0-), \cL, \cF \big)
    = n_-(\Theta - PM_0P) = n_+(PM_0P - \Theta),
  \end{equation}
  since $\cM(0-)$ is the graph of $M_0$ and $\cL$ corresponds to the frame
  $(P,\, P \Theta P + P - I )$.
\end{proof}

\subsection{Comparing Dirichlet and Neumann eigenvalues}

Another easy consequence of our results is a version of Friedlander's well-known 
interlacing formula \cite{Fri_arma91}, which, in our case of finite
defect indices, takes the geometric form of a Duistermaat index.
Defining the \emph{Dirichlet and Neumann extensions} of $\minop$ to be the self-adjoint extensions 
$H_D$ and $H_N$ with
  \begin{equation}
  	\dom(H_D) = \{ f \in \dom(\minop^*) : \Gamma_0 f = 0\}, \quad 
	\dom(H_N) = \{ f \in \dom(\minop^*) : \Gamma_1 f = 0\},
  \end{equation}
  respectively, we get the following.
  
\begin{corollary}
  \label{thm:Friedlander}
  Assume, in addition to the hypotheses of \Cref{thm:main}, 
  that $H_F = H_D$. For any $\lambda \in \R$ below the essential spectrum we have
  \begin{equation}
  \label{eq:Friedlander_formula0}
  \sigma(H_N,H_D;\lambda)
  = \iD\!\big(\cK \oplus 0,\, 0\oplus\cK,\, \cM(\lambda)\big).
\end{equation}
If $\lambda$ is not a Dirichlet eigenvalue, then $\cM(\lambda)$ is the graph of an operator $M(\lambda)$ on $\cK$, and
\begin{equation}
  \label{eq:Friedlander_formula}
  \sigma(H_N,H_D;\lambda) = n_{0+}\big(M(\lambda)\big).
\end{equation}
\end{corollary}

\begin{proof}
Since $H_D$ and $H_N$  correspond to the Lagrangian planes $\cL_D = 0 \oplus \cK$ 
  and $\cL_N = \cK \oplus 0$, the first equation
  follows immediately from \Cref{thm:shift_Hormander}, using \eqref{eq:iD_special_cases} to eliminate
$\iota(\cL_N, \cL_D, \cF)=\iota(\cL_N, \cF, \cF)=0$, and the second follows from \Cref{cor:DNmap} .
\end{proof}

When $\lambda$ is also not a Neumann eigenvalue (and hence $\ker M(\lambda) = 0$), this
is exactly the formula of Friedlander \cite[Lem.~1]{Fri_arma91} (see
also \cite{AreMaz_cpaa12} and references therein) in the context of
finite defect numbers.  Note that the Dirichlet-to-Neumann map in \cite{Fri_arma91} is
$-M(\lambda)$ here, due to the choice of normal derivative\,---\,the abstract Green's identity 
\eqref{eq:Greens_identity} requires $\Gamma_1$ to be the \emph{inward} normal derivative.

\Cref{cor:DNmap} can also be used when $\lambda$ is
a Dirichlet eigenvalue, leading to a more general version of \eqref{eq:Friedlander_formula} in terms of a ``reduced''
Dirichlet-to-Neumann map, as in
\cite{BerCoxHelSun_jga22}.

\subsection{Sharpness of the bounds}
\label{sec:sharp_bounds}

We now prove that the bounds in \eqref{eq:spec_shift_bound} are sharp
for any symmetric operator $\minop$.

\begin{proposition}
  \label{prop:sharp_bounds}
  Let $\minop$ satisfy the assumptions of \Cref{thm:main}. For any
  numbers $\tilde \sigma_\pm \geq 0$ with
  $\tilde\sigma_- + \tilde\sigma_+ \leq n$, there exist Lagrangian
  planes $\cL_1$ and $\cL_2$ such that
  \begin{equation}
    \label{sigmatilde}
    \iD(\cL_1, \cL_2, \cF) = \tilde\sigma_-, \qquad
    \iD(\cL_2, \cL_1, \cF) = \tilde\sigma_+,
  \end{equation}
  and the bounds in \eqref{eq:spec_shift_bound} are sharp 
  for the corresponding extensions $H_1$ and $H_2$ of $\minop$.
  To be precise, there exists $\lambda_0 \in \R$ such that
  \begin{equation}
    \sigma(H_1, H_2; \lambda_0-0) = \tilde\sigma_-, \qquad
    \sigma(H_1, H_2; \lambda_0+0) = \tilde\sigma_+.
  \end{equation}
\end{proposition}

\begin{proof}
  Without loss of generality we assume that our boundary triplet is
  chosen such that $\cF = \cV = 0 \oplus \cK$. Now choose $\lambda_0 \in \R$ 
  below the essential spectrum such that
  $\cM(\lambda_0)$ is transversal to $\cV$.  This is
  always possible because $\cM$ is increasing, therefore its intersections
  with any Lagrangian plane are isolated.  Letting $M_0 \colon \cK \to\cK$ be the
  operator whose graph is $\cM(\lambda_0)$, we define $\cL_1$
  and $\cL_2$ via the frames $(I, M_0)$ and $(I, M_0 - P_- + P_+)$, 
  respectively, 
  where $P_-, P_+ \colon \cK\to\cK$ are arbitrary mutually orthogonal orthogonal
  projectors of rank $\tilde\sigma_-$ and $\tilde\sigma_+$. Note that $\cL_1 = \cM(\lambda_0)$.

  We first show that this choice of $\cL_{1}$ and $\cL_2$ gives the
  desired Duistermaat indices.  Using
  \Cref{prop:BL_index_formula} with $P=I$, we get
  \begin{align}
    \label{eq:sharp_bounds}
    \iD(\cL_1, \cL_2, \cF)
    &= n_-\big((M_0 - P_- + P_+) - M_0\big) = \rank P_- = \tilde\sigma_-, \\
    \iD(\cL_2, \cL_1, \cF)
    &= n_-\big( P_- - P_+\big) = \rank P_+ = \tilde\sigma_+,
  \end{align}
which is exactly \eqref{sigmatilde}.

  Now we recall from the proof of \Cref{thm:main}, in particular 
  \eqref{eq:shift_bounds}, that the lower bound in 
  \eqref{eq:spec_shift_bound} is attained at
  some $\lambda$ if $\iD\!\big(\cL_1, \cL_2, \cM(\lambda)\big) = 0$, and 
  the upper bound is attained if $\iD\!\big(\cL_1,
  \cL_2, \cM(\lambda)\big) = n - \dim \cL_1 \cap \cL_2$.  We then use 
  \eqref{eq:iD_special_cases} and \Cref{prop:continuous} to obtain
  \begin{equation}
    \label{eq:sharp_calc1}
    \iD\!\big(\cL_1, \cL_2, \cM(\lambda_0-0)\big)
    = \iD\!\big(\cM(\lambda_0), \cL_1, \cL_2\big)
    = \iD(\cL_1, \cL_1, \cL_2) = 0
  \end{equation}
 and
  \begin{equation}
    \label{eq:sharp_cals2}
    \iD\!\big(\cL_1, \cL_2, \cM(\lambda_0+0)\big)
    = \iD(\cL_1, \cL_2, \cL_1)
    = n - \dim \cL_1 \cap \cL_2,
  \end{equation}
  completing the proof.
\end{proof}

\begin{example}
  \label{ex:non-sharp}
  It is not true that for any $\cL_1$ and $\cL_2$ the bounds of
  \Cref{thm:main} are achieved at some $\lambda$.  A simple
  example is the Neumann versus the Dirichlet Laplacian on the interval
  $(0,\pi)$.  The spectra are $\{0,1,4, 9,\ldots\}$ and
  $\{1,4,9,\ldots\}$, respectively, so the spectral shift
  takes the values $0$ and $1$ while $\smin=0$ and $\smax=2$.
\end{example}

\subsection{An operator with inner solutions}
\label{sec:noUCP}

Defining the space
\begin{equation}
  \mathcal{D} := \Big\{f = (f_1,f_2) \in
  H^2(-\pi,0) \oplus H^2(0,\pi) : f_1(-\pi) = f_2(\pi) = 0, \  f_1(0) = f_2(0) \Big\},
\end{equation}
we consider the symmetric operator $\minop$ acting as $-\frac{d^2}{dx^2}$ on
$L^2(-\pi,0) \oplus L^2(0,\pi)$, with
\begin{equation}
  \label{eq:dom_noUCP1}
  \dom(\minop) := \Big\{ (f_1,f_2) \in \mathcal D : f_1(0) = f_2(0) = 0, \   f_1'(0) = f_2'(0) \Big\}.
\end{equation}
It easily follows that $\dom(\minop^*) = \mathcal D$. As a boundary triple we can take $\cK = \C$ and
\begin{equation}
  \label{eq:triple_noUCP}
  \Gamma_0 f = \frac12 \big(f_1(0) + f_2(0)\big),
  \qquad \Gamma_1 f = f_2'(0) - f_1'(0).
\end{equation}
  
For every $z \in \C \backslash \{1, 4, 9, \ldots\}$ the kernel
\begin{equation}
  \label{eq:ker_mostly}
  \ker(\minop^*-zI) = \Span\left\{\left(
      \frac1{\sqrt{z}} \sin\sqrt{z}(\pi+x),\,
      \frac1{\sqrt{z}} \sin\sqrt{z}(\pi-x) \right) \right\}
\end{equation}
is one dimensional (with $\ker(S^*) = \Span\{(\pi+x, \pi-x)\} $ understood as 
the $z\to0$ limit).  However, when $z = k^2$ for some 
$k\in\N$ the kernel is two dimensional,
\begin{equation}
  \label{eq:ker_inner}
  \ker(\minop^*-k^2I) = \Span \left\{ \big(
      \sin k(\pi+x),\, 0\big),\ 
    \big(0,\ 
      \sin k(\pi-x) \big) \right\}.
\end{equation}
On the other hand, the Cauchy data space from \eqref{eq:CDS_def} is one dimensional
for all $z \in \C$,
\begin{equation}
  \label{eq:M_noUCP} \cM(z) = \Span \left\{ \left(\frac1{\sqrt{z}}
      \sin\pi\sqrt{z},\, {}-2\cos\pi\sqrt{z}\right) \right\},
\end{equation}
since the trace $\tr$ vanishes on the difference of the two basis
vectors in (\ref{eq:ker_inner}) when $z \in \{1,4,9,\ldots\}$.  In
other words, for any $k\in \mathbb{N}$ the function
\begin{equation}
  \label{eq:inner_solution}
  g_k = \big( \sin k(\pi+x),\, 
    - \sin k(\pi-x) \big)
\end{equation}
is an inner solution in the sense of \eqref{c12}.

Consider, for example, the Lagrangian planes
\begin{equation}
  \label{eq:planes_noUCP} \cF = \Span\left\{
    \begin{pmatrix} 0 \\ 1
    \end{pmatrix} \right\}, \qquad \cL_1 = \Span\left\{
    \begin{pmatrix} 1\\0
    \end{pmatrix} \right\},
\end{equation}
and the corresponding extensions of $\minop$,
\begin{align}
  \label{eq:domF_noUCP} \dom(H_F)
  &:= \big\{ f \in \mathcal{D} \colon
    f_1(0) = f_2(0) = 0 \big\}, \\ \dom(H_1)
  &:= \big\{ f \in \mathcal{D}
    \colon f_1'(0) = f_2'(0) \big\}.
\end{align}
As the notation suggests, $H_F$ is the Friedrichs extension of $\minop$.
The corresponding spectra are easily computed to be
\begin{equation}
  \label{eq:spec_noUCP}
  \spec(H_F) = \left\{ k^2 \text{ with
      multiplicity 2} : k\in\N\right\},
  \qquad
  \spec(H_1) = \left\{
    \left(k/2\right)^2 \colon k\in\N\right\}.
\end{equation}
We now see that if one tries to use the intersections
$\cM(z) \cap \cF$ and $\cM(z) \cap \cL_1$ to search for the
eigenvalues of $H_F$ and $H_1$, every second
eigenvalue will be missed, because these are the eigenvalues that
correspond to inner solutions.

\appendix

\section{Lagrangian preliminaries}
\label{app:Lagrangian}

\subsection{The Lagrangian Grassmannian}
Given a complex symplectic space $(\cK \oplus \cK, \omega)$, we say that a 
subspace $\cL \subset \cK \oplus \cK$ is \emph{Lagrangian} if it equals its \emph{symplectic complement}
 \begin{equation}
	\cL^\omega
      \ :=\  \big\{u \in \cK \oplus \cK \colon \omega(u, v) = 0\
        \text{for all } v \in \cL \big\}.
\end{equation}
There are several convenient reformulations of this definition. These involve the operator
\begin{equation}
	J:=\begin{pmatrix}0&I\\-I&0\end{pmatrix}
\end{equation}
on $\cK \oplus \cK$, defined so that $\omega(u,v) = \left<u, Jv\right>_{\cK \oplus \cK}$. 
It follows that $\cL^\omega = (J\cL)^\perp$, thus $\cL$ is Lagrangian if and only if $J\cL = \cL^\perp$, 
where $(\cdot)^\bot$ is the usual orthogonal complement. A useful consequence of this 
is the identity
\begin{equation}\label{PRLS}
	(\cL_1 + \cL_2)^\perp = \cL_1^\perp \cap \cL_2^\perp = J\cL_1 \cap J\cL_2 = J(\cL_1 \cap \cL_2),
\end{equation}
valid for any Lagrangian planes $\cL_1$ and $\cL_2$. This implies $\dim \cL_1 \cap \cL_2 = \codim(\cL_1 + \cL_2)$. 
In particular, $\cL_1 + \cL_2 = \cK \oplus \cK$ if and only if $\cL_1 \cap \cL_2 = 0$.
Another useful fact is that a subspace $\cL$ is Lagrangian if and only if
\begin{equation}
\label{PJ}
	J = JP_{\cL} + P_{\cL} J,
\end{equation}
where $P_{\cL}$ is the orthogonal projection onto $\cL$; see \cite[Prop.~2.11]{Fur_jgp04}.

The set of all Lagrangian subspaces in $(\cK \oplus \cK, \omega)$ is called the \emph{Lagrangian Grassmannian} and is 
denoted by $\Lagr$.

\subsection{Parameterizations of Lagrangian subspaces}
\label{sec:param}
A Lagrangian subspace $\cL \subset \cK \oplus \cK$ can be described in many different ways. For convenience we summarize the most useful ones here.

A \emph{frame} is an injective linear map $Z \colon \cK \to \cK \oplus \cK$ whose range is $\cL$. A frame is typically written in block form
\begin{equation}
\label{Zdef}
	Z = \begin{pmatrix} X \\ Y \end{pmatrix},
\end{equation}
with $X,Y \colon \cK \to \cK$. We often abbreviate this as $(X,Y)$. The range of $Z$ is Lagrangian if and only if the condition $X^*Y = Y^*X$ holds. This frame is not uniquely determined by $\cL$, since for any $C \in \GL(\cK)$, $ZC$ is also a frame for $\cL$. Any frame for $\cL$ is of this form for a suitable choice of $C$.

A closely related notion is that of a \emph{co-frame}, which is a surjective linear map $\cK \oplus \cK \to \cK$ whose kernel is $\cL$; see \cite[Thm.~1.4.4]{BerKuc_graphs}. This can be written in block form as $\big(A \ \, B \big)$, with $A,B \colon \cK \to \cK$. The Lagrangian condition is equivalent to $AB^* = BA^*$, and it is easily seen that the frame $(X,Y)$ corresponds to the co-frame $\big(Y^* \  -\!X^* \big)$. 

It was shown in 
\cite[Cor.~5]{Kuc_wrm04} (see also \cite[Lem.~2]{BehLug_jpa10}) that any Lagrangian plane $\cL$ can be represented by a co-frame with $A = I - P - P\Theta P$ and $B = P$, or equivalently by the frame
\begin{equation}
	\begin{pmatrix} B^* \\ -A^* \end{pmatrix} = \begin{pmatrix} P \\ P\Theta P + P - I \end{pmatrix},
\end{equation}
where $P \colon \cK \to \cK$ is an orthogonal projection and $\Theta$ is a self-adjoint operator on $\Ran P$. 
This parameterization arises naturally when describing boundary conditions: For $f \in \dom(\minop^*)$
we have $\tr f = (\Gamma_0f, \Gamma_1f) \in \cL$ precisely when
\begin{equation}
	(I-P) \Gamma_0 f = 0, \qquad P \Gamma_1 f = \Theta P \Gamma_0 f,
\end{equation}
so we interpret $P$ and $I - P$ as projections onto the Robin and the Dirichlet parts of $\cK$, respectively.

Another possibility is to write $\cL$ as the graph of an operator defined on a reference Lagrangian subspace. Let $\cL^\sharp$ and $\hat\cL$ be transversal Lagrangian subspaces. If $\cL$ is transversal to $\hat\cL$, then there exists an operator $L \colon \cL^\sharp \to \hat\cL$ whose graph is $\cL$, in the sense that
\begin{equation}
	\cL = \big\{ v + Lv : v \in \cL^\sharp \big\}.
\end{equation}
In particular, if $\hat\cL = (\cL^\sharp)^\perp$, we can write $L = JT$ for some $T \in \cB(\cL^\sharp)$, and the Lagrangian condition is equivalent to $T^* = T$. For instance, if $\cL^\sharp = \cK \oplus 0$ and $\hat\cL = 0 \oplus \cK$, then $\cL$ is transversal to $\hat\cL$ if and only if it has a frame $(X,Y)$ for which $X$ is invertible, and hence an equivalent frame
\[
	\begin{pmatrix} X \\ Y \end{pmatrix} \sim \begin{pmatrix} I_\cK \\ Y X^{-1} \end{pmatrix}.
\]
It therefore corresponds to the graph of $L = JT$ with $T = Y X^{-1} \in \cB(\cL^\sharp)$.

Another possibility is to fix a Lagrangian subspace $\cL_0$ and write $\cL$ as the image of an invertible operator $G_\cL \in \cB(\cK \oplus \cK)$. Such an operator can be explicitly constructed as follows. If $(X,Y)$ is a frame for $\cL$, then
\begin{equation}
\label{GXYdef}
	G^{X,Y} = \begin{pmatrix} X & -Y \\ Y & X \end{pmatrix}
\end{equation}
is invertible and maps the horizontal subspace $\cK \oplus 0$ to $\cL$; cf. \cite[Prop.~1]{Pan_rmp06}. Choosing a frame $(X_0,Y_0)$ for $\cL_0$, we see that $G^{X,Y} \big(G^{X_0,Y_0}\big)^{-1}$ is a valid choice of $G_\cL$.

Finally, we recall that $\cL$ can be written as the graph of a unitary operator 
from $\ker(J-i)$ to $\ker (J+i)$. This can be related to the frame description of $\cL$ by decomposing
\begin{equation}
	\begin{pmatrix} Xu \\ Yu \end{pmatrix} = \underbrace{\frac12 \begin{pmatrix} (X-iY)u \\ (Y + iX)u \end{pmatrix}}_{\in\, \ker(J-i)} 
	+ \underbrace{\frac12 \begin{pmatrix} (X+iY)u \\ (Y - iX)u \end{pmatrix}}_{\,\in \ker(J+i)}
\end{equation}
for arbitrary $u \in \cK$. The desired unitary map is
\begin{equation}
\label{Lunitary}
	\begin{pmatrix} a \\ ia \end{pmatrix} = \frac12 \begin{pmatrix} (X-iY)u \\ (Y + iX)u \end{pmatrix}
	\mapsto
	\frac12 \begin{pmatrix} (X+iY)u \\ (Y - iX)u \end{pmatrix} = \begin{pmatrix} b \\ -ib \end{pmatrix},
\end{equation}
therefore $b = (X+iY)(X-iY)^{-1}a$. (This agree with the formula in \cite[Prop.~1]{Sch_laa12}, since the $\mathcal J$ used there coincides with $-J$ in this paper.) This parameterization will be used in \Cref{sec:path_indices} to define the Maslov index.

\subsection{Smooth structure on the Lagrangian Grassmannian}
\label{app:smooth}
We now recall the smooth structure on $\Lambda$ and give some equivalent formulations for the differentiability of a path $\cL(\cdot) \colon (0,1) \to \Lambda$, in terms of the different parameterizations given above.

To prove that $\Lambda$ is a smooth manifold, one first shows that it is a topological manifold, and then equips it 
with a smooth atlas of coordinate charts. The topology is given by the \emph{gap metric}, $d(\cL_1,\cL_2) := \|P_{\cL_1} - P_{\cL_2} \|$, where $\cL_j \in \Lambda$ and $P_{\cL_j} \in \cB(\cK \oplus \cK)$ are the corresponding orthogonal projections.

Next, given a Lagrangian subspace $\cL$, we let $U_{\cL^\perp}$ denote the set of Lagrangian subspaces that are transversal to $\cL^\perp = J\cL$. This is an open neighborhood of $\cL$ in $\Lambda$, and is homeomorphic to the set $\cB_{\rm sa}(\cL)$ of self-adjoint operators on $\cL$, via the map
\begin{equation}
\label{coordinate}
	A \mapsto \gr_\cL(A) := \{v + JAx : v \in \cL\}.
\end{equation}
See \cite[Prop.~2.21]{Fur_jgp04} for details. It follows that $\Lambda$ is a topological manifold. Moreover, it can be shown that the ``transition functions" are smooth on any intersection $U_{\cL_1^\perp} \cap U_{\cL_2^\perp}$ of coordinate charts. That is, referring to the map $\gr_\cL \colon  \cB_{\rm sa}(\cL) \to \Lambda$ defined in \eqref{coordinate}, we have that the composition
\[
	\gr_{\cL_2}^{-1} \circ \gr_{\cL_1} \colon \ \underbrace{\gr_{\cL_1}^{-1} \big(U_{\cL_1^\perp} \cap U_{\cL_2^\perp}\big)}_{\subseteq \,\cB_{\rm sa}(\cL_1)}  \longrightarrow 
	 \underbrace{\gr_{\cL_2}^{-1} \big(U_{\cL_1^\perp} \cap U_{\cL_2^\perp}\big)}_{\subseteq \,\cB_{\rm sa}(\cL_2)}
\]
is $C^\infty$; see, for instance, \cite[Cor.~2.25]{Fur_jgp04}. This gives $\Lagr$ the structure of a smooth manifold.

Differentiability of a path $\cL(t)$ of Lagrangian subspaces is
defined with respect to this manifold structure.  However, it is often
easier to work with the following equivalent conditions, which are in
fact taken as definitions in some papers.

\begin{theorem}
Given a path $\cL(\cdot): (0,1) \to \Lambda$, the following are equivalent:
\begin{enumerate}
	\item $\cL(t)$ is differentiable;
	\item the family of orthogonal projections $P_{\cL(t)}$ is differentiable;
	\item there exists a differentiable family of invertible
          operators $G_t$ on $\cK \oplus \cK$ such that, for some
          $\cL_0$, $G_t \cL_0 = \cL(t)$;
        \item there exists a differentiable frame $Z(t)$ for $\cL(t)$;
	\item there exists a differentiable family of unitary operators $U_t \colon \ker(J-i) \to \ker(J+i)$ such that 
	$\cL(t) = \{v + U_t v : v \in \ker(J-i)\}$.
\end{enumerate}
\end{theorem}

\begin{proof}
$(1) \Rightarrow (2)$. Suppose $\cL(t)$ is differentiable. In a neighborhood of an arbitrary point $t_0$, $\cL(t)$  is given by $\cL(t) = \{x + J A(t) x : x \in \cL(t_0)\}$ for some differentiable family $A(t)$. Using \cite[Eq.~(2.16)]{Fur_jgp04}, we see that the orthogonal projection $P_{\cL(t)}$ is differentiable.

$(2) \Rightarrow (3)$. We fix an arbitrary $t_0 \in (0,1)$ and then, following \cite[Thm.~IV.1.1]{DaletskiiKrein},
define $G_t$ as the solution to the differential equation
\begin{equation}
  \label{eq:Gt_diffeq}
  \frac{d}{dt} G_t = P_{\cL(t)}'\left(2P_{\cL(t)} - I\right) G_t,
  \qquad G_{t_0}=I.
\end{equation}
Choosing $\cL_0 = \cL(t_0)$, we find that $G_t$ has the desired property (and in addition is unitary).
%


$(3) \Rightarrow (4)$. Defining $G^{X_0,Y_0}$ as in \eqref{GXYdef}, we see that $G_t  G^{X_0,Y_0}$ is differentiable and maps $\cK \oplus 0$ onto $\cL(t)$. Writing this in block form
\[
	G_t  G^{X_0,Y_0} = \begin{pmatrix} X(t) & \ast \\ Y(t) & \ast \end{pmatrix}
\]
in the decomposition $\cK \oplus \cK$, it follows that $\big(X(t),Y(t) \big)$ is a differentiable frame for $\cL(t)$.

$(4) \Rightarrow (1)$. Let $t_0 \in (0,1)$. Using the frame $Z(t)$, any point $v \in \cL(t)$ can be written as
\[
	v = Z(t)u = P_0 Z(t)u + (I - P_0)Z(t)u
\]
for some $u \in \cK$, where $P_0 = P_{\cL(t_0)}$. Since the frame $Z(t_0) \colon \cK \to \cK \oplus \cK$ is injective and has range $\cL(t_0)$, the map $C(t) := P_0 Z(t) \colon \cK \to \cL(t_0)$ is invertible at $t = t_0$, and therefore is invertible for $t$ sufficiently close to $t_0$. It follows that any $v \in \cL(t)$ can be written as
\[
	v = w + (I - P_0) Z(t) C(t)^{-1} w
\]
for some $w \in \cL(t_0)$.
This show that $\cL(t)$ is the graph of the differentiable family $A(t) = -J (I - P_0) Z(t) C(t)^{-1} \colon \cL(t_0) \to \cL(t_0)$, and hence is differentiable at $t_0$.

$(4) \Leftrightarrow (5)$ This is an immediate consequence of \eqref{Lunitary}.
\end{proof}

\subsection{Maslov indices for Lagrangian paths}
\label{sec:path_indices}

For a continuous family $\cM(\cdot):[a,b] \to \Lambda$ of Lagrangian planes, we can represent each  
$\cM(t)$ as the graph of a unitary operator $U_\cM(t) \colon \ker(J - i) \to \ker(J + i)$. 
Doing the same for $\cL(t)$, we obtain a unitary family
\begin{equation}
	W(t) := U_\cM(t) \big(U_\cL(t)\big)^{-1}
\end{equation}
on $\ker(J + i)$ such that $\dim\ker \!\big(W(t) - 1\big) = \dim \!\big(\cM(t) \cap \cL(t)\big)$; see 
\cite[Lem.~2]{BoossZhu2013}. The \emph{Maslov index of $\cM(\cdot)$ with respect to $\cL(\cdot)$} is defined to be the spectral flow of $W(t)$ 
through the point 1 on the unit circle, in the counterclockwise direction. More precisely,
\begin{equation}
	\Mas_{[a, b]}\!\big(\cM(\cdot), \cL(\cdot) \big) := \sum_{j=1}^n \left( \ceil*{\frac{\theta_j(b)}{2\pi}} - \ceil*{\frac{\theta_j(a)}{2\pi}}  \right),
\end{equation}
where $\lceil \cdot \rceil$ denotes the ceiling and $\theta_1, \ldots, \theta_n \colon [a,b] \to \R$ 
are continuous functions such that  $e^{i \theta_1(t)}, \ldots, e^{i \theta_n(t)}$ are the eigenvalues of $W(t)$.
See \cite[\S2.2]{BooZhu_memAMS} or \cite[\S2]{ZhoWuZhu_fmc18} for details.

\begin{remark}
We follow the conventions 
and notation of \cite{ZhoWuZhu_fmc18} so we can directly use their 
formula \eqref{eq:Hor_ZWZ0} relating the Duistermaat and Maslov indices.
Compared to the Maslov index defined by Cappell, Lee and Miller in 
\cite{CapLeeMil_cpam94}, we have
\begin{equation}
\label{Mas_CLM}
	\Mas^{\rm CLM}_{[a, b]}\!\big(\cM(\cdot), \cL(\cdot) \big) = \Mas_{[a, b]}\!\big(\cL(\cdot), \cM(\cdot) \big),
\end{equation}
see the remark in \cite[Def.~A.9]{BarOffPorWu_mz21}. On the other hand, this is related to the Maslov index defined by Robbin and Salamon in \cite{RobSal_t93} by 
\begin{equation}
	\Mas^{\mathrm{RS}}_{[a,b]}\!\big(\cM(\cdot),\cL(\cdot)\big) = \Mas_{[a,b]}\!\big(\cM(\cdot),\cL(\cdot)\big)  + \tfrac12 h(b) - \tfrac12 h(a),
\end{equation}
where we have abbreviated $h(t) := \dim\! \big(\cL(t)\cap\cM(t)\big)$; 
see \cite[Eq.~(A.7)]{BarOffPorWu_mz21}. \hfill $\diamond$
\end{remark}

Comparing \eqref{Mas_CLM} with \cite[Eq.~(1.12)]{CapLeeMil_cpam94}, 
we obtain
\begin{equation}
  \label{eq:CLM_antisymmetric}
  \Mas_{[a,b]}\!\big(\cL(\cdot),\cM(\cdot)\big)
  = -\Mas_{[a,b]}\!\big(\cM(\cdot),\cL(\cdot)\big)  + h(a) - h(b).
\end{equation}
That is, the Maslov index is antisymmetric up to boundary terms.

We now explain how to compute the Maslov index, assuming for the rest of the section that $\cL(t) = \cL$ is constant 
and $\cM(t)$ is differentiable. We say $t_0$ is a \emph{crossing} if $\cM(t_0) \cap \cL \neq 0$. The associated \emph{crossing form} is $\mathfrak m_{t_0} := \mathfrak q|_{\cM(t_0)\cap\cL}$, where $\mathfrak q$ is the form on $\cM(t_0)$ defined by \eqref{Q:def}. A crossing $t_0$ is \emph{regular} if the form $\mathfrak m_{t_0}$ is non-degenerate. For a $C^1$ path $\cM(\cdot)$ with only regular crossings on $[a,b]$, the Maslov index with respect to $\cL$ is then given by
\begin{equation}\label{MIforms}
	\Mas_{[a,b]}\!\big(\cM(\cdot),\cL\big) = n_+(\mathfrak{m}_{a})+\sum_{t_0\in(a,b)}\big(n_+(\mathfrak{m}_{t_0})-n_-(\mathfrak{m}_{t_0})\big)-n_-(\mathfrak{m}_{b}).
\end{equation}
Regular crossings of a $C^1$ path are isolated, so the sum over $t_0$ is finite. 
For complex symplectic spaces \eqref{MIforms} was proved in \cite[Prop.~3.27]{BooZhu_memAMS}. For real spaces this method of computing the Maslov index first appeared in \cite{RobSal_t93}. In this paper we only require the special case when $\cM(\cdot)$ is an increasing path, hence all crossing forms are positive definite; see \eqref{eq:Maslov_intersections}.


\bibliography{bk_bibl, mybib_analysis_prop}

 \bibliographystyle{siam}


\end{document}